\documentclass[10pt]{article}
\usepackage{amsmath,amssymb,amsthm}
\usepackage{color}
\usepackage[nohead,margin=1.0in]{geometry}
\usepackage{booktabs}
\usepackage{graphicx,caption,subcaption}
\usepackage{url}
\usepackage[normalem]{ulem}

\theoremstyle{plain}
\newtheorem{theorem}{Theorem}[section]
\newtheorem{remark}{Remark}[section]

\newtheorem{lemma}{Lemma}[section]
\newtheorem{assumption}{Assumption}[section]
\newtheorem{proposition}{Proposition}[section]

\numberwithin{equation}{section}

\newcommand{\N}{\mathbb{N}}
\newcommand{\R}{\mathbb{R}}

\renewcommand{\d}{\mathrm{d}}

\title{On the Degree of Ill-Posedness of Multi-Dimensional Magnetic Particle Imaging}
\author{Tobias Kluth\thanks{Center for Industrial Mathematics, University of Bremen,
Bibliothekstr. 5, 28357 Bremen, Germany ({tkluth@math.uni-bremen.de})} \and
Bangti Jin\thanks{Department of Computer Science, University College London, Gower Street,
London WC1E 6BT, UK ({b.jin@ucl.ac.uk, bangti.jin@gmail.com})}
\and Guanglian Li\thanks{Department of Mathematics, Imperial College London, London SW7 2AZ, UK. The work was partially carried out when the
author was affiliated with Institut f\"ur Numerische Simulation and Hausdorff Center for Mathematics, Universit\"at Bonn, Wegelerstra\ss e
6, D-53115 Bonn, Germany. (lotusli0707@gmail.com, guanglian.li@imperial.ac.uk).}
}
\begin{document}
\maketitle

\begin{abstract}
Magnetic particle imaging is an imaging modality of relatively recent origin, and it exploits
the nonlinear magnetization response for reconstructing the concentration of nanoparticles.
Since first invented in 2005, it has received much interest in the literature.
In this work, we study one prototypical mathematical model in multi-dimension, i.e., the equilibrium model, which
formulates the problem as a linear Fredholm integral equation of the first kind. We analyze the
degree of ill-posedness of the associated linear integral operator by means of the singular
value decay estimate for Sobolev smooth bivariate functions, and discuss the influence of various
experimental parameters on the decay rate. In particular, applied magnetic fields with a field free point and a
field free line are distinguished. The study is complemented with extensive numerical experiments.\\
{\bf Keywords}: magnetic particle imaging; degree of ill-posedness; equilibrium model; singular
value decay; Sobolev smooth bivariate functions.
\end{abstract}

\section{Introduction}

Magnetic particle imaging (MPI) is a relatively new imaging modality \cite{Gleich2005}.
The main goal is to reconstruct a spatially dependent concentration of iron oxide nanoparticles by
exploiting their superparamagnetic behavior. Measurements are obtained from multiple receive coils where
a voltage is induced by particles' nonlinear response to the applied dynamic magnetic field using
either field free point (FFP) \cite{Gleich2005} or field free line (FFL) \cite{Weizenecker2008ffl} trajectories.
These measurements can yield reconstructions with a high spatial/temporal resolution.
Since the modality is free from harmful radiation, it is especially beneficial for in-vivo applications.

So far, MPI has been used for preclinical medical applications, and holds a significant potential for clinical applications.
One application, already suggested at the beginning of the MPI development, is vascular imaging \cite{Gleich2005}.
In in-vivo experiments, the potential for imaging blood flow was demonstrated using healthy mice \cite{weizenecker2009three}.
Recently, it was studied for long-term circulating tracers \cite{khandhar2017evaluation}.
The high temporal resolution allows tracking medical instruments \cite{haegele2012magnetic},
e.g., in angioplasty \cite{Salamon:2016cz}. Other potential applications include
cancer detection \cite{Yu2017} and cancer treatment by hyperthermia \cite{murase2015usefulness}.

In practice, MPI is usually modeled by a linear Fredholm integral equation of the first kind.
This is motivated by the suppression of particle interactions due to nonmagnetic coating,
which allows postulating a linear relationship between particle concentration and the measured voltage.
However, precisely modeling MPI respectively formulating a physically accurate integral
kernel for image reconstruction is still an unsolved problem due to various modeling errors in the
particle dynamics and data acquisition, e.g., magnetization dynamics, particle-particle interactions
and transfer function for analog filter; we refer interested readers to the survey paper \cite{Kluth:2017}
for further details. In the literature, the equilibrium model based on the Langevin function has
been used extensively to predict the signal behavior in MPI \cite{Knopp2012,Knopp2017}; see
Section \ref{sec:prelim} below for details on the model, its derivation and the underlying assumptions.

The mathematical study on the MPI model is fairly scarce. The only work on FFP trajectories
that we are aware of is the work \cite{maerz2015modelbased}.
In a 1D setup with FFP trajectories moving along a line for the equilibrium model, the authors
\cite{maerz2015modelbased} showed that in the limit of large particle diameters, the integral
kernel is a Dirac-delta function, and thus the imaging problem is well-posed.
Further, they analyzed a related problem which is independent of the FFP trajectory used,
under the assumption that each spatial point is scanned multiple times
with nonparallel trajectories. The problem was shown to be severely ill-posed in general, and
in the large particle diameter limit, the smoothing property of the forward operator
improves with the spatial dimension $d$. Also, if a FFL is moved by a drive field in its
perpendicular direction, the problem can be formulated using Radon transform, followed by
a convolution with a kernel involving the mean magnetic moment \cite{Knopp2011FFLfourierslice}.
However, the theoretical analysis of general FFL trajectories remains missing.

In this work, we present a study on the degree of ill-posedness of the MPI inverse problem.
Historically, the idea of distinguishing mildly, moderately and severely ill-posed problems can be
traced at least back to Grace Wahba \cite{Wahba:1980}.
Since the 1980s, the concept ``degree of ill-posedness'' for linear inverse problems has been
popular. Roughly, it refers to the decay behavior of the singular values (SVs) $\sigma_n$:
 $\sim n^{-\nu}$ with small $0<\nu<1$ for mildly ill-posed problems, with $1\leq
\nu <\infty$ for moderately ill-posed ones, and $e^{-\nu n}$ with $\nu>0$ for severely
ill-posed ones. There are at least two reasons to look at the SV decay rate (see, e.g.,
\cite{DickenMaass:1996,HofmannScherzer:1994}). First, it characterizes the degree of
ill-posedness of the imaging problem, which is one factor in determining the
resolution limit of the image reconstruction step,  when the data accuracy and model accuracy
etc. are given. Second, the analysis also sheds insight into how to improve the resolution by properly
changing the experimental setting. Hence, over the past few decades, SV decay estimates have
received much interest for a number of inverse problems, especially in the context of computed
tomography.

In this work, we study the degree of ill-posedness of MPI via SV decay of the associated linear
integral operator. Our analysis relies crucially on the SV decay estimate for Sobolev smooth
bivariate functions \cite{GriebelLi:2017}, and its
extension to less regular bivariate functions. The extension seems still unavailable and will
be given in this work, and the result is of independent interest. Our results give upper bounds
on SV decay rates, which indicate the (best possible) degree of ill-posedness of the MPI model.
We discuss the following three  cases separately: nonfiltered model, limit model and
filtered model, in order to illuminate the influences of experimental parameters, e.g., particle
size and the regularity of the analog filter. Further, we conduct extensive numerical experiments
to complement the analysis. When completing the paper, the authors became aware of the work
\cite{ErbWeinmannAhlborg:2017}, where Erb et al. showed the exponential ill-posedness of a 1D
MPI model. It differs substantially from this work in the main focus (multi-dimensional model) and analytical tools.

Note that the SV decay is only one factor for the ``level of recovery chances''
to linear ill-posed problems, as already emphasized by Louis
\cite{Louis:1989} in 1989, with the other being ``solution smoothness with respect
to the character of the forward operator''. Only both factors and
their interplay allow realistic error estimates on the reconstructions from noisy data, and the
degree of ill-posedness must be put into the context of regularization.
Such an analysis is beyond the scope of this work; see
the works \cite{MatheHofmann:2008,HofmannKindermann:2010} for relevant results.

The remainder of the paper is organized as follows. In Section \ref{sec:prelim}, we
describe the equilibrium model, and in Section \ref{sec:svd},
the SV decay estimate for Sobolev smooth bivariate functions.
Then in Section \ref{sec:anal}, we analyze the decay rate
for the integral operators for nonfiltered and filtered models, and
discuss the influence of various factors, e.g., spatial
dimensionality $d$ and particle parameter $\beta$. In Section \ref{sec:numer}, we present
numerical results to support the analytical findings, and finally, in Section \ref{sec:conc},
we give concluding remarks. Throughout, we denote by $C$ with/without subscript a generic constant
which may differ at each occurrence.

\section{The equilibrium model}\label{sec:prelim}

Now we describe the equilibrium model, one prototypical mathematical model for MPI.

\subsection{Preliminaries}
MPI is inherently a 3D problem, and thus vector valued functions remain 3D even if the
domain $\Omega$ of the spatial variable $x$ is a subset of a $d$-dimensional affine
subspace $E_d\subset \R^3$. Let $\Omega \subset E_d$, $d=1,2,3$, be a bounded domain
with a (strong) Lipschitz boundary $\partial\Omega$ in $E_d$. Further, let $T>0$ denote the maximal
data acquisition time and $I:=(0,T)$ the time interval during which the measurement
process takes place. The temporal derivative of any function $g: I \rightarrow \R^d$ is
denoted by $\dot{g}$.

In MPI, the measured signal $v_\ell : I \rightarrow \R$, $\ell=1,\hdots,L$, obtained at
$L\in \N$ receive coils, is given by
\begin{equation}
\label{eq:complete-problem}
 v_\ell (t) = \int_I \int_\Omega c(x) {a}_\ell(t-t') \kappa_\ell(x,t) \d x \d t' +
\underset{=v_{\mathrm{E},\ell}(t)}{\underbrace{\int_I \int_{\R^3} {a}_\ell(t-t')\mu_0 p_\ell(x)^t\dot{ H}(x,t)  \d x \d t'}},
\end{equation}
where the superscript $t$ denotes the transpose of a vector,
$c: \Omega \rightarrow \R^+ \cup \{0\}$ is the concentration of the magnetic nanoparticles and
$\kappa_\ell:\Omega \times I \rightarrow \R$, $\ell=1,\hdots,L$, represent the system functions characterizing
the magnetization behavior of nanoparticles. The
positive constant $\mu_0$ is magnetic permeability in vacuum.
The scalar functions $a_\ell: \bar{I}:=[-T:T] \rightarrow \R$, $\ell=1,\hdots,L$, are the analog
filter in the signal acquisition chain, and in practice,
they are often band stop filters adapted to excitation frequencies of the drive field
so as to minimize the adverse influence of the excitation signal $v_{\text{E},\ell}$ during digitization.
The functions $p_\ell: \R^3 \rightarrow \R^3$, $\ell=1,\hdots, L$, denote the vector field which characterizes
the sensitivity profile of the receive coils and can be spatially dependent.
Throughout, it is assumed that the applied magnetic field $H: \R^3 \times I \rightarrow \R^3$ and
the filters $\{a_\ell\}_{\ell=1}^L$ are chosen in a way such that all excitation signals $v_{\mathrm{E},\ell}=0$, $\ell=1,\ldots,L$.

\begin{remark}
The assumption on the excitation signals $\{v_{\mathrm{E},\ell}\}_{\ell=1}^L$ is commonly made,
which, however, may be not fulfilled in MPI
applications \cite{Them2016,Kluth2017}. Note that in the model \eqref{eq:complete-problem}, we have absorbed the minus
sign into the measurement to make the notation more consistent with literature on integral equations.
\end{remark}

The applied magnetic field $H(x,t)$ can be characterized by a spatially dependent
magnetic field $g:\R^3 \rightarrow \R^3$ and a time-dependent homogeneous magnetic field $h:I \rightarrow \R^3$,
and the field $H(x,t)$ is given by their superposition, i.e., $H(x,t) = g(x) -h(t)$. The field $g$, named
{\it selection field}, ensures that a field-free-region is generated. Generally, $g$ is assumed to
be linear such that it can be represented by a constant matrix $G\in \R^{3\times 3}$.
The field $h$, named {\it drive field}, then moves the field-free-region along a certain trajectory.
Two MPI methodologies are distinguished by the field free region, whether a FFP is generated ($\mathrm{rank}(G)
=3$) or a FFL is used ($\mathrm{rank}( G) =2$). For the FFL approach, it was also proposed  to
rotate the selection field $g$ over time such that the FFL is rotated \cite{Weizenecker2008ffl}, and then the selection field
is given by $g: \R^3 \times I \rightarrow \R^3$ with $g(x,t)=P(t)^tGP(t)x$ where $P(t):I\rightarrow
\R^{3\times 3}$ is a rotation matrix for all $t\in I$.

The functions $\{\kappa_\ell\}_{\ell=1}^L$ can be expressed using the receive coil sensitivities
$\{p_\ell\}_{\ell=1}^L$ and the particles' mean magnetic moment vector $\bar{m}:\Omega \times I
\rightarrow \R^3$ as $\kappa_\ell=\mu_0p_\ell^t\dot{\bar m}$.
This relation follows from Faraday's law and the law of reciprocity \cite{Knopp2012}.
Then the inverse problem is to find the concentration $c:\Omega\to \mathbb{R}^+\cup\{0\}$
from $\{v_\ell\}_{\ell=1}^L$:
\begin{equation} \label{eq:general-problem}
 v_\ell(t) =  \int_I \int_\Omega c(x) {a}_\ell(t-t') \kappa_\ell(x,t') \d x \d t', \quad\mbox{with }
\kappa_\ell = \mu_0 p_\ell^t \dot{\bar{m}}.
\end{equation}

In the model \eqref{eq:general-problem}, the linear dependence on the concentration $c$ is derived under
the assumption that particle-particle interactions can be neglected. However, there is experimental evidence that
these interactions can affect the particle signal \cite{loewa2016}.

\subsection{Equilibrium MPI model}
\label{sec:particle_behavior}

To specify the MPI model, it remains to describe the mean magnetic moment vector $\bar m(x,t)$
of nanoparticles. There are several possible models, e.g., Fokker-Planck equation or stochastic Landau-Lifschitz-Gilbert equation
\cite{Kluth:2017}. The most extensively studied
model in MPI is based on the assumptions that the applied magnetic field $H(x,t)$ is static, the particles are
in equilibrium, and $\bar m(x,t)$ immediately follows the magnetic field $H(x,t)$. Then by Langevin theory for
paramagnetism, $\bar{m}(x,t)$ is given by
\begin{equation*}
 \bar{m}(x,t) = m_0\mathcal{L}_\beta(| H(x,t)|) \frac{H(x,t)}{| H(x,t)|},
\end{equation*}
where the parameter $m_0$ is particle's magnetic moment, $|\cdot|$ denotes the Euclidean norm of vectors,
and $\mathcal{L}_\beta: \R \rightarrow \R$ is the (scaled) Langevin function given by
\begin{equation}\label{eq:Langevin}
 \mathcal{L}_\beta(z) =  \coth( \beta z) - (\beta z)^{-1},
\end{equation}
where $\beta$ is a given positive parameter. The Langevin function $\mathcal{L}_\beta(z)$ captures the
nonlinear response to the applied magnetic field. The resulting model is termed as {\it equilibrium model} below.

\begin{remark}
Physically, the parameters $m_0$ and $\beta$ are determined by the saturation magnetization
$M_{\mathrm{S}}$ of the core material, the volume $V_{\mathrm{C}}$ of the single-domain particle's core,
the temperature $T_\mathrm{B}$, and Boltzmann constant $k_\mathrm{B}$, i.e.,
$m_0=M_{\mathrm{S}}V_{\mathrm{C}}$ and $\beta=\mu_0 m_0/(k_\mathrm{B} T_\mathrm{B})$.
Note that $\beta$ also depends on the particle diameter $D$ through $m_0$. At room temperature
$293 \text{ K}$, particles consisting of magnetite with a typical diameter $D$ of $30 \text{ nm}$ $(20 \text{ nm})$
are characterized by $\beta \approx 2.1 \times 10^{-3}$ $(0.6 \times 10^{-3})$.
\end{remark}

The function $L_\beta(z)$ is a smooth approximation to the sign function: for any
fixed $0<\beta<\infty$, it belongs to $C^\infty(\mathbb{R})$, and as $\beta\to\infty$, it
recovers the sign function $\mathrm{sign}(z)$.

\begin{lemma}\label{lem:langevin}
The Langevin function $L_\beta(z)$ has the following properties: $\rm(i)$ For any $z\in\mathbb{R}$,
there holds $\lim_{\beta\to\infty}\mathcal{L}_\beta(z) = {\rm sign}(z)$; and
$\rm(ii)$ The function $\frac{L_\beta(\sqrt{z})}{\sqrt{z}}\in C_B^\infty([0,\infty))$.
\end{lemma}
\begin{proof}
First, recall the following expansions for the $\coth(z)$ :
\begin{align}\label{eqn:Langevin-small}
 \coth(z) = \frac{1}{z} + \frac{z}{3}-\frac{z^3}{45}+\frac{2}{945}z^5-\ldots+\frac{2^{2n}B_{2n}}{(2n)!}z^{2n-1} + \ldots\quad (|z|<\pi),
\end{align}
where $B_n$ is the $n$th Bernoulli number \cite[4.5.67, p. 85]{AmbramowitzStegun:1964}.
Meanwhile, for any $z>0$, we have
\begin{equation}\label{eqn:Langevin-large}
  \coth(z) = \frac{e^{z}+e^{-z}}{e^z-e^{-z}} = 1 + \frac{2e^{-2z}}{1-e^{-2z}} = 1+2\sum_{j=1}^\infty e^{-2jz},
\end{equation}
and a similar series expansion holds for $z<0$. Now assertion (i) follows directly
from \eqref{eqn:Langevin-large}. For part (ii), it suffices to show $\beta=1$.
For $0\leq z<\pi^2$, it follows from
\eqref{eqn:Langevin-small} and the definition of $L_1(z)$ that
\begin{equation*}
  \frac{L_1(\sqrt{z})}{\sqrt{z}} = \frac{1}{3}-\frac{z}{45}+\frac{2}{945}z^2-\ldots+\frac{2^{2n}B_{2n}}{(2n)!}z^{n-1} + \ldots,
\end{equation*}
which is smooth in $z$ (and convergent for any $0\leq z<\pi^2$). The assertion
 for $z>1$ follows from \eqref{eqn:Langevin-large}.
\end{proof}

In summary, the inverse problem for the equilibrium model is to recover
the concentration $c$ from
\begin{equation}\label{eqn:mpi-general}
\left\{
\begin{aligned}
 v_\ell(t) &=  \int_I \int_\Omega c(x) {a}_\ell(t-t') \kappa_\ell(x,t') \d x \d t', \\
 \kappa_\ell(x,t) & = \mu_0m_0 p_\ell^t \frac{\d}{\d t}\left[ \frac{\mathcal{L}_\beta(| H|)}{|H|} H \right],
\end{aligned}
\right.
\end{equation}
for $\ell=1,\hdots,L$ and the magnetic field $H:\Omega  \times I \rightarrow \R^3$ is given by
$H(x,t)= g(x) - h(t),$ where $g: \Omega  \rightarrow \R^3$  and $h: I \rightarrow \R^3$.
For MPI, a common choice is $g(x)=G x$ and $h(t)= A (\sin(f_i t) )_{i=1}^3$, where
$A\in \R^{3\times 3}$ is a diagonal matrix with $1\leq \text{rank}(A)\leq d$, $f_i>0$,
and a constant matrix $G\in \R^{3\times 3}$ with $\mathrm{tr}(G)=0$. Below, for the matrix $G$, we
distinguish two cases: (i) $G$ has full rank such that a FFP is generated, and
(ii) $\mathrm{rank}(G)=2$ such that a FFL is generated (only for 3D).

In the model \eqref{eqn:mpi-general}, all vectors belong to $\mathbb{R}^3$, which
reflects the intrinsic 3D nature of the MPI imaging problem. However, for a spatial domain $\Omega\subset
\mathbb{R}^d$, $d=1,2$, the dimensionality of the vectors and matrices can be taken to be $d$, by
properly restricting to subvectors/submatrices, which is feasible under the assumption that
$G^{-1}h(t)\in \Omega$, for any $t\in I$ (i.e., field free region is contained in $\Omega$). Specifically, the $d$-dimensional case,
for $d=1,2$, can be constructed by assuming that the concentration $c$ is a Dirac $\delta$-distribution
with respect to the orthogonal complement of the affine subspace $E_d\subset \R^3$, i.e., $c(x)= c_d(x_1)
\delta(x_2)$, where $x=x_1+x_2$ with $x_1\in E_d$, $x_2 \in E_d^\perp$, and $c_d: \Omega \subset E_d
\rightarrow \R^+ \cup \{0\}$. The parametrization of the domain $\Omega\subset E_d$ then allows reformulating
the integral in \eqref{eqn:mpi-general} in terms of $\Omega_d \subset \R^d$. Given the affine linear
parametrization $\Gamma : \Omega_d \rightarrow \Omega$, \eqref{eqn:mpi-general} can be
stated with respect to $\tilde{c}_d: \Omega_d \rightarrow \R^+ \cup \{0 \}$, $\tilde{c}_d(x)
=c_d(\Gamma(x))$. This convention will be adopted in the analysis
below, by directly writing $h(t)\in I\to \mathbb{R}^d$ etc.

\section{Singular value decay for Sobolev smooth bivariate functions}\label{sec:svd}
Now we describe SV decay for Sobolev smooth bivariate functions,
which is the main technical tool for studying degree of ill-posedness in Section \ref{sec:anal}.

\subsection{Preliminaries on function spaces}
First, we recall Sobolev spaces and Bochner-Sobolev spaces, which are used extensively below.
For any index $\alpha\in \mathbb{N}^d$, $|\alpha|$ is the sum of all components. Given a domain $D
\subset\mathbb{R}^d$ with a Lipschitz continuous boundary, for any $m\in \mathbb{N}$, $1\leq p\leq
\infty$, we follow \cite{AdamsFournier:2003} and define the Sobolev space $W^{m,p}(D)$ by
\begin{equation*}
W^{m,p}(D)=\big\{u\in L^{p}(D): D^{\alpha}u\in L^{p}(D) \text{ for } 0\leq
	|\alpha|\leq m\big\}.
\end{equation*}
It is equipped with the norm
\begin{equation*}
  \|u\|_{W^{m,p}(D)} = \left\{\begin{aligned}
    \Big(\sum\limits_{0\leq |\alpha|\leq m}\|D^{\alpha}u\|_{L^p(D)}^{p}\Big)^{\frac{1}{p}}, &\quad \text{ if }1\leq p<\infty,\\
    \max\limits_{0\leq |\alpha|\leq m}\|D^{\alpha}u\|_{L^\infty(D)} ,&\quad \text{ if } p=\infty.
  \end{aligned}\right.
\end{equation*}
The space $W_{0}^{m,p}(D)$ is the closure of $C^{\infty}_{0}(D)$ in $W^{m,p}(D)$. Its dual space
is denoted by $W^{-m,p'}(D)$, with $\frac{1}{p}+\frac{1}{p'}=1$, i.e., $p'$ is the conjugate exponent of
$p$. Also we use $H^{m}(D)=W^{m,2}(D)$, and $H_0^m(D)=W_0^{m,2}(D)$. The fractional order Sobolev space $W^{s,p}(D)$,
$s\geq0, s\notin \mathbb{N}$, can be defined by interpolation \cite{AdamsFournier:2003}.
It can be equivalently defined by a Sobolev--Slobodecki\v{\i} seminorm $|\cdot|_{W^{s,p}(D)}$. For $0<s<1$, it is defined by
\begin{equation}\label{eqn:SS-seminorm}
   | u |_{W^{s,p}(D)}^p := \int_D\int_D \frac{|u(x)-u(y)|^p}{|x-y|^{d+sp}} \,\d x\d y ,
\end{equation}
and the full norm $\|u\|_{W^{s,p}(D)}=(\|u\|_{L^p(D)}^p + |u|_{W^{s,p}(D)}^p)^\frac{1}{p}$.
For $s>1$, it can be defined similarly.

We state a result on pointwise multiplication on Sobolev spaces \cite[Theorem 7.5]{BehzadanHolst:2017}.
\begin{theorem}\label{thm:product}
Let $D\subset\mathbb{R}^d$, $d=1,2,3$. Assume that $s_i,s$ $(i=1,2)$ are real numbers satisfying
$s_i\geq s \geq 0$ and $s_1+s_2-s>\frac {d}{2}$. Then for some constant $C(s_1,s_2,s,d)$, there holds
\begin{equation*}
  \|uv\|_{H^{s}(D)}\leq C(s_1,s_2,s,d)\|u\|_{H^{s_1}(D)}\|v\|_{H^{s_2}(D)}\quad \forall u\in H^{s_1}(D),v\in H^{s_2}(D).
\end{equation*}
\end{theorem}

Suppose $X$ is a Banach space, with the norm denoted by $\|\cdot\|_X$. Then, for any $p\in\mathbb{N}$,
we denote by $H^p(I;X)$ the Bochner space of functions $v:I\rightarrow X$ such that $v(t)$ and
its weak derivatives (in time) up to order $p$, i.e., $\dot v(t),\ldots, v^{(p)}(t)$, all exist and
belong to $L^2(I;X)$. The norm on $H^p(I;X)$ is defined by
\begin{equation*}
 \|v\|_{H^p(I;X)} ^2 = \sum_{j=0}^p\int_I \|v^{(j)}(t)\|_{X}^2\d t.
\end{equation*}
Then for any $s\geq0$, we can define $H^s(I;X)$ by means of interpolation, and equivalently
using the Sobolev-Slobodecki\v{\i} seminorm \cite{HytonenWeis:2016}. For example, for $s\in(0,1)$,
then the seminorm $|\cdot|_{H^s(I;X)}$ is defined by
\begin{equation*}
    |v|_{H^s(I;X)}^2 = \int_I\int_I\frac{\|v(t_1,\cdot)-v(t_2,\cdot)\|_{X}^2}{|t_1-t_2|^{1+2s}}\d t_1\d t_2,
\end{equation*}
and $\|v\|_{H^s(I;X)}=(\|v\|_{L^2(I;X)}^2+|v|_{H^s(I;X)}^2)^\frac{1}{2}$.
We shall use the case $X=L^2(D)$ extensively. The space $H^s(0,T;L^2(D))$ is
isomorphic to $H^s(0,T)\times L^2(D)$ and $L^2(D;H^s(0,T))$, i.e.,
$H^s(0,T;L^2(D))\simeq H^s(0,T)\times L^2(D)\simeq L^2(D;H^s(0,T))$ (see, e.g.,
\cite[Proposition 1.2.24, p. 25]{HytonenWeis:2016} for the isomorphism
$L^2(I;L^2(D))\simeq L^2(D;L^2(I))$, from which the general case may be derived).
Then by \cite[Proposition 1.3.3, p. 39]{HytonenWeis:2016}, we have
$L^2(D;H^{-s}(I))\simeq H^{-s}(I;L^2(D))$.
We shall use these isomorphisms frequently below.

Now we give two results on the composition operator on $H^s(I;L^\infty(D))$
and $H^s(D;W^{1,\infty}(I))$.
\begin{lemma}\label{lem:compos-Linf}
The following two statements hold.
\begin{itemize}
  \item[$\rm(i)$]For $v\in H^s(I;L^\infty(D))\cap L^\infty(I;L^\infty(D))$ $(s\geq0)$ and $g\in C_B^k(\mathbb{R})$ $(k\geq [s]+1)$, $g\circ v \in H^s(I;L^\infty(D))$.
  \item[$\rm(ii)$] For $v\in H^s(D;W^{1,\infty}(I))\cap L^\infty(D;W^{1,\infty}(I))$ $(s\geq0)$ and $g\in C_B^k(\mathbb{R})$ $(k\geq [s]+2)$, $g\circ v \in H^s(D;W^{1,\infty}(I))$.
\end{itemize}
\end{lemma}
\begin{proof}
By the definition \eqref{eqn:SS-seminorm} and the mean value theorem,
for $0<s<1$, since $v\in L^\infty(I;L^\infty(D))$, we have
\begin{align*}
  |g\circ v|_{H^s(I;L^\infty(D))}^2 & = \int_I\int_I\frac{\|g(v(t_1,\cdot))-g(v(t_2,\cdot))\|_{L^\infty(D)}^2}{|t_1-t_2|^{1+2s}}\d t_1\d t_2 \\
     & \leq \int_I\int_I\frac{\sup_{\xi\in\mathbb{R}}|g'(\xi)|^2\|v(t_1,\cdot)-v(t_2,\cdot)\|_{L^\infty(D)}^2}{|t_1-t_2|^{1+2s}}\d t_1\d t_2 \\
     &\leq \sup_{\xi\in\mathbb{R}}|g'(\xi)|^2|v|_{H^s(I;L^\infty(D))}^2<\infty.
\end{align*}
The case $s\geq1$ follows similarly by the chain rule and Theorem \ref{thm:product}.
For example, for $1<s<2$, by chain rule, $\frac{\d}{\d t}(g\circ v)(t)=(g'\circ v)\dot{v}$.
Since $\dot{v} \in H^{s-1}(I;L^\infty(\Omega))$ and $g'\circ v\in H^1(I;L^\infty(\Omega))$,
Theorem \ref{thm:product} implies $\frac{\d}{\d t} (g\circ v)\in H^{s-1}(I;L^\infty(\Omega))$, showing
the assertion for $s\in (1,2)$.

For any $v\in L^\infty(D;W^{1,\infty}(I))$, by the chain rule, mean value theorem and triangle inequality,
direct computation gives that for any $x_1,x_2\in D$
\begin{align*}
 \|g(v(\cdot,x_1))-&g(v(\cdot,x_2))\|_{W^{1,\infty}(I)}
 =  \|g(v(\cdot,x_1))-g(v(\cdot,x_2))\|_{L^\infty(I)} \\
  &\qquad + \|g'(v(\cdot,x_1))\dot{v}(\cdot,x_1) -g'(v(\cdot,x_2))\dot{v}(\cdot,x_2)\|_{L^\infty(I)}\\
  & \leq C\|(v(\cdot,x_1)-v(\cdot,x_2)\|_{W^{1,\infty}(I)} + C\|v(\cdot,x_1)-v(\cdot,x_2)\|_{L^\infty(I)}\|\dot{v}(\cdot,x_1)\|_{L^\infty(I)},
\end{align*}
where the constant $C$ depends only on $\|g\|_{C^2_{B}(\mathbb{R})}$.
By the definition \eqref{eqn:SS-seminorm}, for $0<s<1$, we have
\begin{align*}
  |g\circ v|_{H^s(D;W^{1,\infty}(I))}^2 & = \int_D\int_D\frac{\|g(v(\cdot,x_1))-g(v(\cdot,x_2))\|_{W^{1,\infty}(I)}^2}{|x_1-x_2|^{d+2s}}\d x_1\d x_2 \\
     & \leq C\int_D\int_D\frac{\|v(\cdot,x_1)-v(\cdot,x_2)\|_{W^{1,\infty}(I)}^2}{|x_1-x_2|^{d+2s}}\d x_1\d x_2 \\
     &\quad +C\int_D\int_D\frac{\|v(\cdot,x_1)-v(\cdot,x_2)\|_{L^\infty(I)}^2\|\dot{v}(\cdot,x_1)\|_{L^\infty(I)}^2}{|x_1-x_2|^{d+2s}}\d x_1\d x_2 \\
     &\leq C(|v|_{H^s(D;W^{1,\infty}(I))}^2 + |v|_{H^s(D;L^\infty(I))}^2\|v\|_{L^\infty(D;L^\infty(I))}^2)<\infty.
\end{align*}
This shows the assertion for $0<s< 1$. The case $s\geq1$ follows similarly as part (i).
\end{proof}

\subsection{Singular value decay}

Now we describe our main tool of the analysis, i.e., SV decay estimates for Sobolev smooth bivariate functions.
The study of eigenvalues of integral operators with a kernel function has a rather long history. The monographs
\cite{Pietsch:1987} and \cite{Konig:1986} contain a wealth of relevant results. However, the results in these
works are concerned with two variables defined on the same domain, which do not handle the integral kernel
$\kappa(x,t)$ directly. We shall use the recent result due to Griebel and Li \cite{GriebelLi:2017} (see
\cite[Theorem 3.2]{GriebelLi:2017}), for the nonfiltered model in Section \ref{ssec:nonfilter}; see
\cite[Chapter 2]{Pietsch:1987} for an introduction to the Lorentz sequence space $\ell_{p,w}$.

\begin{theorem}\label{thm:svd-decay-reg}
Suppose that $D\subset\mathbb{R}^d$ satisfies the strong local Lipschitz condition. Let
$\kappa(x,y)\in L^2(\Omega,H^s(D))$, $s\geq 0$. Then the SVs $\sigma_n$ of the
associated integral operator satisfy
\begin{equation*}
  \sigma_n \leq \mathrm{diam}(D)^s C_{\rm em}(d,s)^\frac{1}{2}C_{\rm ext}(D,s)^\frac{1}{2}\|\kappa\|_{L^2(\Omega,H^s(D))}n^{-\frac{1}{2}-\frac{s}{d}},
\end{equation*}
where the constant $C_{\rm ext}(D,s)$ depends only on $D$ and $s$ {\rm(}for Sobolev extension{\rm)}, $C_{\rm em}(d,s)$
is an embedding constant for $\ell_{\frac{d}{d+2s},1}\hookrightarrow\ell_{\frac{d}{d+2s},\infty}$,
and $\mathrm{diam}(D)$ is the diameter of the domain $D$.
\end{theorem}

Note that the result in Theorem \ref{thm:svd-decay-reg} requires $s\geq 0$, which does not cover less regular
kernels for the limit model in Section \ref{ssec:limit} below. For general rough kernels, the
spectral theory is largely open \cite{Pietsch:1987,Konig:1986}. Below we analyze the kernel $f:D\times I\to\mathbb{R}$ 
defined by (with $d>1$ being the dimension of the domain $D$)
\begin{equation*} 
  f(x,t):=|Gx-h(t)|^{-\frac d2}.
\end{equation*}

We will make the following assumption.
\begin{assumption}\label{ass:h}
Suppose that the matrix $G\in\mathbb{R}^{d\times d}$ is invertible, and
the trajectory $h(t):I\to \mathbb{R}^d$ satisfies 
\begin{itemize}
\item[$\rm(i)$] There exists $C_h$ such that
$C_h:=\sup_{t\in[0,T]}|\dot{h}(t)|^{-1}<\infty.$
\item[$\rm(ii)$]For any $t\in I$, $G^{-1}h(t)\in D$; and
there exists at most $N_h$ distinct $t\in I$ such that $h(t)=Gx$, for any $x\in D$.
\end{itemize}
\end{assumption}
\begin{remark}
The condition $G^{-1}h(t)\in D$ describes that the FFP moves within the
physical domain $D$, and the domain $D$ is properly covered by the trajectories.
The analysis remains valid if the condition holds
for any open subinterval of $I$. If for all $t\in I$,
$G^{-1}h(t)\notin D$, then $|Gx-h(t)|^r$ belongs to
$C^\infty(\overline{I};L^2(D))$ for smooth trajectories $h(t)$, and the analysis in
Section \ref{ssec:nonfilter} applies directly.
\end{remark}

First, by the proof of Lemma \ref{lem:reg1} below, we have $f(x,t)\in L^2(I;L^p(D))$ for any $p\in(1,2)$.
Now we define the associated integral operator $\mathcal{S}: L^{p'}(D)\rightarrow  L^{2}(I)$ with $f(x,t)$
as its kernel, and its adjoint operator $\mathcal{S}^{*}: L^{2}(I)\rightarrow  L^{p}(D)$, respectively, by
\begin{align*}
(\mathcal{S}v)(t)=\int_{D}f(x,t)v(x){\d}x \quad \mbox{and}\quad
(\mathcal{S}^{*}v)(x)=\int_If(x,t)v(t){\d}t.
\end{align*}
Let $\mathcal{R}: L^{p'}(D)\rightarrow L^{p}(D)$ by $\mathcal{R}=\mathcal{S}^{*}\mathcal{S}.$
By construction, $\mathcal{R}$ is an integral operator with its kernel
$R\in L^p(D)\times L^p(D):D\times D\to \mathbb{R}$ given by
\begin{align}\label{eq:kernelR}
R(x,x')=\int_If(x,t)f(x',t)\d t.
\end{align}

Now we give mapping properties of $\mathcal{S}$.
See \cite[pp. 221--228]{AdamsFournier:2003} for an introduction to Lorentz spaces $L^{p,\infty}(D)$.
\begin{lemma}\label{lem:bdd-kernel}
For $d=2,3$, let Assumption \ref{ass:h} be fulfilled. Then the following statements hold.
\begin{itemize}
\item[$\rm(i)$] For all $q>2$, the operator $\mathcal{S}:L^q(D)\to L^\infty(I)$ is bounded;
\item[$\rm(ii)$] For all $p\in (\frac{2(d-1)}{d}, 2)$ and $q=\frac{2p}{d(2-p)}$, the operator
$\mathcal{S}:L^p(D)\to L^q(I)$ is bounded. In addition, $\mathcal{S}$ is compact from $L^2(D)$ to $L^2(I)$.
\end{itemize}
\end{lemma}
\begin{proof}
By the proof of Lemma \ref{lem:reg1} below, $v(t):=\|f(\cdot,t)\|_{L^{2-\epsilon}(D)}\in L^{\infty}(I)$ for
any $\epsilon>0$. Together with \cite[Theorem 6.1, p. 99]{Krasnoselskii:1976}, it implies
that $\mathcal{S}$ is bounded from $L^q(D)$ to $L^{\infty}(I)$ for all $q>{2}$.

The proof of assertion (ii) is inspired by the proof of Theorem 8.10 of \cite[p. 165]{Krasnoselskii:1976}.
It relies on Stein-Weiss interpolation theorem (see, e.g., \cite[Theorem 8.2, p. 150]{Krasnoselskii:1976} and
\cite[Chapter 5]{MR0304972}). First, we claim that there is $C>0$ such that for
any $p\in (\frac{2(d-1)}{d}, 2)$ and $q=\frac{2p}{(2-p)d}$, there holds
\begin{align}\label{eq:stein-weiss}
\|\mathcal{S}\chi_{A}\|_{L^{q,\infty}(I)}\leq C|A|^\frac{1}{p}
\end{align}
for all measurable subsets $A\subset D$ with finite measure, where $\chi_A$ is the characteristic
function of $A$. If the estimate \eqref{eq:stein-weiss} holds, then by Stein-Weiss interpolation theorem,
$\mathcal{S}$ is bounded from $L^p(D)$ to $L^q(I)$ for all $p\in (\frac{2(d-1)}{d}, 2)$ and
$q=\frac{2p}{d(2-p)}$. Then Theorem 5.4 of \cite[p. 83]{Krasnoselskii:1976} implies that $\mathcal{S}$
is compact from $L^2(D)$ to $L^2(I)$. Hence, it suffices to prove the estimate \eqref{eq:stein-weiss}.

First, by H\"{o}lder's inequality with an exponent $\gamma\in (\frac{1}{q},\frac{1}{p})$, we obtain
\begin{align}
\mathcal{S}\chi_{A}(t)&=\int_{A}|Gx-h(t)|^{-\frac{d}{2}}\d x=\int_{A}|Gx-h(t)|^{-(\frac{d}{2}-\frac{1}{q})}|Gx-h(t)|^{-\frac{1}{q}}\d x\nonumber\\
&\leq \Big(\int_{A}|Gx-h(t)|^{-(\frac{d}{2}-\frac{1}{q})(1-\gamma)^{-1}}\d x\Big)^{1-\gamma}\Big(\int_{A}|Gx-h(t)|^{-\frac{1}{q\gamma}}\d x\Big)^{\gamma}.\label{eq:3344}
\end{align}
Let $s:=(\frac{d}{2}-\frac{1}{q})(1-\gamma)^{-1}$, and let $B(G^{-1}h(t),\rho):=\{x\in \mathbb{R}^{d}:
|x-G^{-1}h(t)|\leq \rho\}$ be the ball centered at $G^{-1}h(t)$ with a radius $\rho$ satisfying $|B(G^{-1}
h(t),\rho)|=|A|$, which implies $|A|=|\mathbb{S}^{d-1}|\rho^d$, where $|\mathbb{S}^{d-1}|$ denotes
the volume of the unit sphere in $\mathbb{R}^d$. Then by equation (8.33) of
\cite[p. 154]{Krasnoselskii:1976}, we have
\begin{align*}
 &\int_{A}|x-G^{-1}h(t)|^{-s}\d x-\int_{B(G^{-1}h(t),\rho)}|x-G^{-1}h(t)|^{-s}\d x\\
=&\int_{A\backslash B(G^{-1}h(t),\rho)}|x-G^{-1}h(t)|^{-s}\d x-\int_{B(G^{-1}h(t),\rho)\backslash A}|x-G^{-1}h(t)|^{-s}\d x\\
\leq &|A\backslash B(G^{-1}h(t),\rho)|\rho^{-s}-|B(G^{-1}h(t),\rho)\backslash A| \rho^{-s}=0.
\end{align*}
Consequently,
\begin{align}\label{eq:2233}
\int_{A}|x-G^{-1}h(t)|^{-s}\d x\leq\int_{B(G^{-1}h(t),\rho)}|x-G^{-1}h(t)|^{-s}\d x.
\end{align}
By the choice of $\gamma$, $s:=(\frac{d}{2}-\frac{1}{q})(1-\gamma)^{-1}=d\frac{p-1}{p(1-\gamma)}<d$.
Thus, in view of the inequality $|Gx-h(t)|\geq
\|G^{-1}\|^{-1}|x-G^{-1}h(t)|$, together with \eqref{eq:2233}, changing to polar coordinates leads to
\begin{align*}
  \int_{A}|Gx-h(t)|^{-s}\d x\leq& \|G^{-1}\|^{s}\int_{B(G^{-1}h(t),\rho)}|x-G^{-1}h(t)|^{-s}\d x\\
\leq& \|G^{-1}\|^{s}|\mathbb{S}^{d-1}|\int_0^{\rho} r^{-s} r^{d-1}\d r\\
\leq & \|G^{-1}\|^{s}|\mathbb{S}^{d-1}|\frac{p(1-\gamma)}{d(1-p\gamma)}\rho^{d\frac{1-p\gamma}{p(1-\gamma)}}.
\end{align*}
Therefore, with the identity $|A|=|\mathbb{S}^{d-1}|\rho^d$, we arrive at
\begin{align*}
\int_{A}|Gx-h(t)|^{-s}\d x\leq C(s,d,\gamma,p)|A|^{\frac{1-p\gamma}{p(1-\gamma)}}.
\end{align*}
with a constant $C(s,d,\gamma,p)=\|G^{-1}\|^{s}|\mathbb{S}^{d-1}|^{\frac{p-1}{p(1-\gamma)}}\frac{p(1-\gamma)}{d(1-p\gamma)}$. This and \eqref{eq:3344} imply
\begin{align}\label{eq:7788}
\|\mathcal{S}\chi_{A}\|_{L^{q,\infty}(I)}\leq 
C(s,d,\gamma,p)^{1-\gamma}|A|^{\frac{1}{p}-\gamma}\Big\|\Big(\int_{A}|Gx-h(t)|^{-\frac{1}{q\gamma}}\d x\Big)^{\gamma}\Big\|_{L^{q,\infty}(I)}.
\end{align}
It remains to estimate the last term in \eqref{eq:7788}. Actually, by definition,
\begin{align}
\Big\|\Big(\int_{A}|Gx-h(t)|^{-\frac{1}{q\gamma}}\d x\Big)^{\gamma}\Big\|_{L^{q,\infty}(I)}&=\Big(\Big\|\int_{A}|Gx-h(t)|^{-\frac{1}{q\gamma}}\d x\Big\|_{L^{q\gamma,\infty}(I)}\Big)^{\gamma}\nonumber\\
&\leq \Big(\int_{A}\Big\||Gx-h(t)|^{-\frac{1}{q\gamma}}\Big\|_{L^{q\gamma,\infty}(I)}\d x\Big)^{\gamma}\nonumber\\
&=\Big(\int_{A}\Big\||Gx-h(t)|^{-1}\Big\|_{L^{1,\infty}(I)}^{\frac{1}{q\gamma}}\d x\Big)^{\gamma}.\label{eq:8899}
\end{align}
Next we fix any $x\in A$, and estimate $\big\||Gx-h(t)|^{-1}\big\|_{L^{1,\infty}(I)}$. Let $g(t):=|Gx-h(t)|^{-1}$.
Then, under Assumption \ref{ass:h}, the nonincreasing rearrangement function $g^*(\tau)$ for $\tau\geq 0$ can
be bounded by
\begin{align}\label{eq:111222}
 g^*(\tau)\leq 2C_hN_h\tau^{-1},
\end{align}
where the rearrangement function $g^*(\tau)$ is defined by
\begin{align*}
g^*(\tau)&=\inf\{c:\big|\{t:|g(t)|\geq c\}\big|\leq \tau\}\\
 &=\inf\{c:\big|\{t:|h(t)-Gx|\leq c^{-1}\}\big|\leq \tau\},
\end{align*}
by slightly abusing the notation $|\cdot|$ for the Lebesgue measure of a set. Indeed, we have the trivial inclusion
$\{t:|h(t)-Gx|\leq c^{-1}\}\subset \cup_{j=1}^{N_x}\{t\in [\max(0,t_{j-1}),\min(t_{j+1},T)]: |h(t)-h(t_j)|\leq c^{-1}\}$,
where the time instances $t_j$ satisfy $h(t_j)=Gx$, for $j=1,\ldots, N_x\leq N_h$, under Assumption \ref{ass:h}.
Further, for any $t\in [t_{j-1},t_{j}]$, by the mean value theorem, there exists some $\xi_j\in [t_{j-1},t_{j}]$ such that
$|h(t)-h(t_j)|=|\dot{h}(\xi_j)(t-t_j)|$, then the assertion \eqref{eq:111222} follows. Consequently
\begin{align*}
\big\||Gx-h(t)|^{-1}\big\|_{L^{1,\infty}(I)}:=\sup\limits_{\tau>0}\tau g^{*}(\tau)\leq 2C_hN_h.
\end{align*}
Now plugging this into \eqref{eq:8899} yields
\begin{align*}
 \Big\|\Big(\int_{A}|Gx-h(t)|^{-\frac{1}{q\gamma}}\d x\Big)^{\gamma}\Big\|_{L^{q,\infty}(I)}\leq (2C_hN_h)^\frac{1}{q}|A|^{\gamma},
\end{align*}
which, together with \eqref{eq:7788}, directly implies
\begin{align*}
\|\mathcal{S}\chi_{A}\|_{L^{q,\infty}(I)}\leq (2C_hN_h)^\frac{1}{q}C(s,d,\gamma,p)^{1-\gamma}|A|^\frac{1}{p}.
\end{align*}
Upon letting $\gamma=\frac{p+q}{2pq}$, we obtain the desired estimate.
This completes the proof of the lemma.
\end{proof}

The next lemma shows that the operator $\mathcal{R}:L^q(D)\to L^2(D)$ is compact, for any $q>\frac{2(d-1)}{d}$.
\begin{lemma}\label{lemma:compact}
Under the conditions of Lemma \ref{lem:bdd-kernel}, for any $q>\min(1,\frac{2(d-1)}{d})$, $\mathcal{R}$
extends to a compact operator from $L^q(D)$ to $L^2(D)$. Especially, $\mathcal{R}$ is compact on $L^2(D)$.
\end{lemma}
\begin{proof}
This follows directly from Lemma \ref{lem:bdd-kernel} and a duality argument.
\end{proof}

By Lemma \ref{lemma:compact}, the operator $\mathcal{R}$ is nonnegative, compact and self-adjoint on $L^2(D)$.
By spectral theory for compact operators \cite{yosida78}, it has
at most countably many discrete eigenvalues, with zero being the only accumulation point, and each nonzero
eigenvalue has only finite multiplicity. Let $\{\lambda_n\}_{n=1}^{\infty}$ be the sequence of eigenvalues (with
multiplicity counted) associated to $\mathcal{R}$, which are ordered nonincreasingly, and $\{\phi_n\}_{n=1}^\infty$
the corresponding eigenfunctions (orthonormal in $L^2(D)$). Moreover, spectral theory implies
\begin{align}\label{eq:spectral}
\forall v\in L^2(D): \quad \mathcal{R}v=\sum\limits_{n=1}^{\infty}\lambda_n(v,\phi_n)\phi_n,
\end{align}
with $(\cdot,\cdot)$ being the $L^2(D)$ inner product. Let $q^*=\infty$ for $d=2$, and $q^*=4$ for $d=3$.
Then by Lemma \ref{lem:bdd-kernel} and \cite[Theorem 5.4, p. 83]{Krasnoselskii:1976},
the eigenfunctions $\{\phi_n\}_{n=1}^{\infty}$ have the following summability:
For any $q<q^*$ and any $n\in \mathbb{N}_{+}$, $\phi_n\in L^q(D)$.
This and the spectral decomposition \eqref{eq:spectral} imply that
the spectrum of the operator $\mathcal{R}$ will not change if its domain is restricted to
$L^{2+\epsilon}(D)$ for any $0<\epsilon<1$.

Now we extend Theorem \ref{thm:svd-decay-reg} to the kernel $|Gx-h(t)|^{-\frac d2} $.
This result will be used in Section \ref{ssec:limit}. We need a few concepts from
spectral theory in Banach spaces \cite{Pietsch:1987}.
Given two Banach spaces $E$ and $F$, the $n$-th approximation number $a_n(W)$ and the Weyl number
$x_n(W)$ of an operator $W\in \mathcal{B}(E,F)$ (i.e., the set of all bounded linear operators from $E$ to $F$) are defined by
\begin{align*}
a_n(W):&=\inf\{\|W-L\|_{\mathcal{B}(E,F)}: L\in \mathfrak{F}(E,F),\text{ rank}(L)<n\},
\end{align*}
and
\begin{align*}
x_n(W):&=\sup\{a_n(WX): X\in \mathcal{B}(\ell_2,E),\|X\|_{\mathcal{B}(\ell_2,E)}\leq 1\},
\end{align*}
respectively, where $\mathfrak{F}(E,F)$ denotes the set of the finite rank operators and $WX$ is the
product of the two operators $W$ and $X$. Furthermore, the following multiplicative property on Weyl
numbers holds \cite[Sections 2.4 and 3.6.2]{Pietsch:1987}:
\begin{proposition}\label{prop:Weyl}
For all  $n\in\mathbb{N}_{+}$, $X\in \mathcal{B}(E_0,E)$,
$W\in \mathcal{B}(E,F)$ and $Y\in\mathcal{B}(F,F_0)$, there holds
\begin{align*}
x_n(YWX)\leq \|Y\|_{\mathcal{B}(F,F_0)}x_n(W) \|X\|_{\mathcal{B}(E_0,E)}.
\end{align*}
\end{proposition}

\begin{theorem}\label{thm:svd-decay}
Let Assumption \ref{ass:h} hold, and $f(x,t):=|Gx-h(t)|^{-\frac d2}$ with $d=2,3$.
Then $\lambda_n \leq C n^{-\frac12+\epsilon}$ as $n\to\infty$ for any $\epsilon>0$.
\end{theorem}
\begin{proof}
Let $\widetilde{\mathcal{S}}:=\mathcal{S}|_{L^{2+\epsilon}(D)}$ for some small fixed $\epsilon>0$, i.e., the restriction of $\mathcal{S}$
on $L^{2+\epsilon}(D)$. Then by Lemma \ref{lem:bdd-kernel}(i), the range of $\mathcal{S}$ is $L^{\infty}(I)$. Hence, we can
decompose $\widetilde{S}:L^{2+\epsilon}(D)\to L^{2+\epsilon}(I)$ into $\widetilde{\mathcal{S}}=\mathcal{I}
\widetilde{\mathcal{S}}$, where $\mathcal{I}$ is the embedding operator from $L^{\infty}(I)$ to $L^{2+\epsilon}(I)$.
The multiplicative property of Weyl numbers $x_n$ in Proposition \ref{prop:Weyl} implies
\begin{align*}
x_n(\mathcal{I}\tilde{\mathcal{S}})\leq x_n(\mathcal{I})\|\tilde{\mathcal{S}}\|_{\mathcal{B}(L^{2+\epsilon}(D),L^{\infty}(I))}.
\end{align*}
By \cite[Section 6.3.4, p. 250]{Pietsch:1987}, there holds $x_n(\mathcal{I})\le C n^{-\frac{1}{2+\epsilon}}$.
Thus, we arrive at
\begin{align*}
x_n(\mathcal{I}\tilde{\mathcal{S}})\le C n^{-\frac{1}{2+\epsilon}}.
\end{align*}
Meanwhile, Lemma \ref{lem:bdd-kernel}(ii) and a standard duality argument indicate that the dual
operator $\mathcal{S}^{*}$ is bounded from $L^{2+\epsilon}(I)$ to $L^{2+\epsilon}(D)$. Note that $\widetilde{\mathcal{R}}
={\mathcal{R}}|_{L^{2+\epsilon}(D)}$, and its eigenvalues are $\{\lambda_n\}_{n=1}^{\infty}$, which can be
bounded by the Weyl numbers $x_n(\widetilde{\mathcal{R}})$ according to the eigenvalue theorem
for Weyl operators \cite[Section 3.6.2]{Pietsch:1987}. Hence, we deduce
\begin{align*}
\lambda_{n}\le C x_{n}(\widetilde{\mathcal{R}})= Cx_{n}({\mathcal{S}}^{*}\mathcal{I}\tilde{\mathcal{S}})
\leq C\|{\mathcal{S}}^{*}\|_{\mathcal{B}(L^{2+\epsilon}(I),L^{2+\epsilon}(D))}x_n(\mathcal{I}\tilde{\mathcal{S}}).
\end{align*}
where the last step is due to Proposition \ref{prop:Weyl}. Combining the preceding
estimates completes the proof.
\end{proof}

\begin{remark}\label{rem:conjecture}
The bound in Theorem \ref{thm:svd-decay} seems not sharp. The sharp one
is conjectured to be $O(n^{-1+\epsilon})$. The statement remains
valid if the kernel $f(x,t)$ is multiplied by a bounded function.
This fact will be used below in Section \ref{ssec:limit}.
\end{remark}

\section{Degree of ill-posedness}\label{sec:anal}

Now we analyze the degree of ill-posedness of the equilibrium MPI model \eqref{eqn:mpi-general}
via the SV decay rate of the associated integral operator, and focus on three cases: (a)
nonfiltered equilibrium model, (b) limit model, and (c) filtered equilibrium model.
Dependent of the problem setting, the behavior of the forward operator can differ substantially
\cite{Weizenecker2007}. Our analysis below sheds
insights into these observations. Since the experimental parameters
for all the receive coils are comparable in practice, our analysis below focuses on one receive coil,
which allows us to simplify the notation. The decay rate given below only determines the best possible degree
of ill-posedness (i.e., upper bounds on SVs), and the results should be only used as an
indicator of the degree of ill-posedness.

\subsection{The non-filtered equilibrium model}\label{ssec:nonfilter}

First, we consider the case in the absence of the temporal analog filter $a(t)$, and discuss
the influence of the filter in Section \ref{ssec:filter} below.  Then the MPI forward operator
$F: L^2(\Omega) \rightarrow L^2(I)$ is given by
\begin{equation}
\label{eqn:non-filtered}
\left\{\begin{aligned}
  v(t) & =  \int_\Omega c(x) \kappa(x,t) \d x,   \\
  \kappa   &= \mu_0m_0 p^t \frac{\d}{\d t} \left[ \frac{\mathcal{L}_\beta(|H|)}{|H|}H\right],  \\
  H(x,t) &= g(x) - h(t).
\end{aligned}\right.
\end{equation}

Now we can state our first main result.
\begin{theorem}\label{thm:non-filter}
Let $0< \beta < \infty$, $d=1,2,3$, $h\in (H^s(I))^d$ with $s\geq 1$, $g\in (L^\infty(\Omega))^d$,
and $p \in (L^\infty(\Omega))^d$. Then for the operator $F:L^2(\Omega)\to L^2(I)$ defined in \eqref{eqn:non-filtered},
the SVs $\sigma_n$ decay as $\sigma_n\leq Cn^{\frac12 - s}$.
\end{theorem}
\begin{proof}
Since $h\in (H^s(I))^d$ and $g\in (L^\infty(\Omega))^d$, by Sobolev embedding, the function
$H(x,t)=g(x)-h(t)\in H^s(I;(L^\infty(\Omega))^d)\subset L^\infty(I;(L^\infty(\Omega))^d)$, and $\dot H(x,t)=-\dot{h}(t)
\in (H^{s-1}(I))^d$. Clearly, we have
\begin{equation}\label{eqn:p-H}
  p^tH \in H^s(I;L^\infty(\Omega)) \quad\mbox{and}\quad
p^t\dot H = -p^t\dot h \in H^{s-1}(I;L^\infty(\Omega)).
\end{equation}
Further, by Lemma \ref{lem:langevin}, $\frac{L_\beta(\sqrt{z})}{\sqrt{z}}\in C_B^\infty([0,\infty))$
and since $s\geq 1$, by Theorem \ref{thm:product}, simple computation shows $|H|^2 \in H^s(I;L^\infty(\Omega))$,
and thus by Lemma \ref{lem:compos-Linf}(i), $\frac{L_\beta(|H|)}{|H|}\in H^s(I;L^\infty(\Omega))$ and
$\frac{\d}{\d t}\frac{L_\beta(|H|)}{|H|}\in H^{s-1}(I;L^\infty(\Omega))$.  Now, by the product rule,
\eqref{eqn:p-H} and Theorem \ref{thm:product}, we deduce
\begin{align}
  \kappa & = \mu_0m_0p^t \frac{\d}{\d t}\left(\frac{L_\beta(|H|)}{|H|}H\right) \nonumber\\
    & = \mu_0m_0\left(p^tH \frac{\d}{\d t}\frac{L_\beta(|H|)}{|H|} +
  \frac{L_\beta(|H|)}{|H|} p^t\dot H\right)\label{eqn:kernel}\\
   & \in H^{s-1}(I;L^\infty(\Omega))\subset H^{s-1}(I;L^2(\Omega)).\nonumber
\end{align}
Then the desired assertion follows from Theorem \ref{thm:svd-decay-reg}.
\end{proof}

\begin{remark}
For $p,g\in (L^\infty(\Omega))^d$, Theorem \ref{thm:non-filter} describes the potential influence of the trajectory $h(t)$
on the SV decay. For smooth trajectories, i.e., $h(t)\in (C^\infty(I))^d$ {\rm(}e.g., sinusoidal trajectory, common in
experimental setup{\rm)}, the SVs decay rapidly, and thus the inverse problem is very ill-posed.
For nonsmooth trajectories, i.e., triangular trajectory {\rm(}$h(t)\in (H^{\frac{3}{2}-\epsilon}(I))^d$, for any
$\epsilon\in(0,\frac12)${\rm)}, the decay may be slower.
\end{remark}

In Theorem \ref{thm:non-filter}, we assume $p,g\in (L^\infty(\Omega))^d$ only. It does not account for possible
additional regularity of $\kappa(x,t)$ in the spatial variable $x$. In practice, it is often taken to be
homogeneous/linear, and thus $\kappa(x,t)$ is very smooth in $x$. This extra regularity can significantly
affect the SV decay, which is described next.

\begin{theorem}\label{thm:non-filter2}
Let $0< \beta < \infty$, $d=1,2,3$, $h\in (H^{s}(I))^d$ with $s>\frac32$, $g\in (H^r(\Omega))^d$, $r>\frac{d}{2}$
and $p \in (C^\infty(\Omega))^d$. Then for the operator $F:L^2(\Omega)\to L^2(I)$ defined in \eqref{eqn:non-filtered}, the SVs
$\sigma_n$ decay as $\sigma_n\le Cn^{-\frac{1}{2}-\frac{r}{d}}$.
\end{theorem}
\begin{proof}
By Sobolev embedding, for $s>\frac32$, $\dot H(x,t) = -\dot h(t) \in (H^{s-1}(I))^d\subset  (L^\infty(I))^d$. Then under
the given assumptions, $H(x,t)=g(x)-h(t)\in H^r(\Omega;(H^s(I))^d)\subset H^r(\Omega;(W^{1,\infty}(I))^d)\subset
L^\infty(\Omega;(W^{1,\infty}(I))^d)$. By Lemma \ref{lem:langevin}, $\frac{L_\beta(\sqrt{z})}{\sqrt{z}}\in C_B^\infty
([0,\infty))$, and by Theorem \ref{thm:product}, we deduce $|H(x,t)|^2\in H^r(\Omega;W^{1,\infty}(I))$.
Hence, Lemma \ref{lem:compos-Linf}(ii) implies $\frac{L_\beta(|H(x,t)|)}{|H(x,t)|}\in H^r(\Omega;W^{1,\infty}(I))$,
and $\frac{\d}{\d t}\frac{L_\beta(|H(x,t)|)}{|H(x,t)|} \in H^r(\Omega;L^\infty(I))$. Further, for $r>\frac d2$,
for small $\epsilon>0$, we have $2r -\frac d2-\epsilon >r$, and thus it follows from \eqref{eqn:p-H} and Theorem
\ref{thm:product} that $p^tH\frac{\d}{\d t}\frac{L_\beta(|H(x,t)|)}{|H(x,t)|}\in H^{r}(\Omega;L^\infty(I))$ and
similarly  $\frac{L_\beta(|H(x,t)|)}{|H(x,t)|}p^t\dot h(t)\in H^r(\Omega;L^\infty(I))$. These two inclusions and
\eqref{eqn:kernel} show that $\kappa\in H^{r}(\Omega; L^\infty(I))\subset H^{r}(\Omega;L^2(I))$. Thus, by Theorem
\ref{thm:svd-decay-reg}, the SVs of the adjoint operator decay as $O(n^{-\frac{1}{2}-\frac{r}{d}})$. Since the adjoint
operator $F^*:L^2(I)\to L^2(\Omega)$(with respect to the $L^2(I\times \Omega)$ inner product) shares the SVs
\cite[p. 27, eq. (2.1)]{GohbergKrein:1969}, the desired assertion follows.
\end{proof}

By Theorem \ref{thm:non-filter2}, the SVs can decay fast for a nonsmooth trajectory $h(t)$, so long as $p(x)$ and
$g(x)$ are sufficiently smooth. The regularity requirement might be relaxed by analyzing more precisely
pointwise multiplication in Bochner-Sobolev spaces.
Since $\dot h(t)$ and $p(x)$ enter the kernel $\kappa(x,t)$ as pointwise multipliers, if uniformly bounded,
they act as bounded operators on $L^2(I)$ and $L^2(\Omega)$, respectively, and the decay rate remains valid
\cite[p. 27, eq. (2.2)]{GohbergKrein:1969}.

\subsection{Limit model}\label{ssec:limit}

It was reported that the spatial resolution increases with particle diameter $D$ \cite{Weizenecker2007,
knopp2008singular}, i.e., a large $\beta$ value in the model \eqref{eqn:non-filtered}. Hence, we analyze
the limit case $\beta \rightarrow \infty$ below. First, we derive the expression for the limit integral
operator. Throughout this part, in Assumption \ref{ass:h}, the domain $D$ refers to $\Omega$.

\begin{proposition}\label{lem:limit-model}
Let $h\in (H^s(I))^d$ with $s\geq 1$, $g\in (L^\infty(\Omega))^d$, and $p \in (L^\infty(\Omega))^d$.
For $\beta\to\infty$, there holds
\begin{equation*}
  \frac{\d}{\d t}\left(\mu_0 p^t  L_\beta(|H|)\frac{H}{|H|}\right) \to \frac{\d}{\d t}\left(\mu_0 p^t \frac{H}{|H|}\right)\quad \mbox{ in } H^{-1}(I;L^2(\Omega)),
\end{equation*}
and the limit integral operator $\widetilde F$ of the operator $F$ defined in \eqref{eqn:non-filtered} is given by
\begin{equation}\label{eqn:problem-limit-case}
\left\{
\begin{aligned}
  v(t) & =  \int_\Omega c(x) \tilde{\kappa}(x,t) \d x,   \\
  \tilde{\kappa} & = \mu_0m_0 p^t \left(-\frac{H H^t}{|H|^3} + \frac{1}{|H|} I_d \right) \dot{H}, \\
  H(x,t) &= g(x) - h(t).
\end{aligned}
\right.
\end{equation}
\end{proposition}
\begin{proof}
First, by the assumptions on $g$, $h$ and $p$ and Sobolev embedding theorem, $H(x,t)=g(x)-h(t)\in (L^\infty
(\Omega\times I))^d$. Then for any $\phi(t)\in H_0^1(I)$ and $\psi(x)\in L^2(\Omega)$, integration by parts yields
\begin{align*}
  &\int_I\int_\Omega \mu_0m_0 p(x)^t\frac{\d}{\d t}\left[L_\beta(|H(x,t)|)\frac{H(x,t)}{|H(x,t)|}\right] \phi(t)\psi(x)\d x\d t \\
  = &-\int_I\int_\Omega\mu_0m_0p(x)^tL_\beta(|H(x,t)|)\frac{H(x,t)}{|H(x,t)|} \dot\phi(t)\psi(x)\d x\d t.
\end{align*}
Now it follows from Cauchy-Schwarz inequality that
\begin{align*}
  &\Big|\int_I\int_\Omega [L_\beta(|H(x,t)|)-{\rm sign}(|H(x,t)|)]\mu_0m_0p(x)^t\frac{H(x,t)}{|H(x,t)|} \dot\phi(t)\psi(x)\d x\d t\Big|\\
  \leq&\mu_0m_0\|p\|_{(L^\infty(\Omega))^d}\int_I\int_\Omega\Big|L_\beta(|H(x,t)|)-{\rm sign}(|H(x,t)|)\Big|\times|\dot\phi(t)\psi(x)|\d x\d t\\
  \leq&\mu_0m_0\|p\|_{(L^\infty(\Omega))^d}\Big\|L_\beta(|H(x,t)|)-{\rm sign}(|H(x,t)|)\Big\|_{L^2(\Omega\times I)}\|\dot\phi\|_{L^2(I)}\|\psi\|_{L^2(\Omega)}.
\end{align*}
Since $\|H\|_{L^\infty(\Omega\times I)}<\infty$, by Lemma \ref{lem:langevin}(i) and
Lebesgue's dominated convergence theorem, we deduce
\begin{equation*}
  \lim_{\beta\to\infty}\|L_\beta(|H(x,t)|)-{\rm sign}(|H(x,t)|)\|_{L^2(\Omega\times I)}=0.
\end{equation*}
By integration by parts again, and density of the product $\psi\phi$ in $H_0^1(I;L^2(\Omega))\simeq L^2(\Omega;H_0^1(I))$
(see, e.g., \cite[Section 6.2, p. 244]{Pietsch:1987} or \cite[Lemma 1.2.19, p. 23]{HytonenWeis:2016}), we obtain the assertion.
\end{proof}

Next we analyze the decay rate of the SVs of the limit operator $\widetilde F$. First, we
give a result on the Bochner-Sobolev regularity of the function $|Gx-h(t)|^{r}$, $r>-\frac d2$.
\begin{lemma}\label{lem:reg1}
Let Assumption \ref{ass:h} hold, and let $h(t)$ be sufficiently smooth.
Then for any $r>-\frac d2$, the function $f(x,t)=|Gx-h(t)|^{r} \in H^s(I;L^2(\Omega))$ for any $s<r+\frac d2$.
\end{lemma}
\begin{proof}
Let $\epsilon>0$ be sufficiently small and $r=-\frac d2+\frac \epsilon2$. Since $G$ is invertible, by the relation
$|Gx-h(t)|=|G(x-G^{-1}h(t))|\geq \|G^{-1}\|^{-1}|x-G^{-1}h(t)|$ for any fixed $t\in [0,T]$, we have
\begin{equation*}
 \int_I\int_\Omega |Gx-h(t)|^{-d+\epsilon} {\d}x{\d}t
\leq\|G^{-1}\|^{d-\epsilon}\int_{0}^{T}\Big(\int_{\Omega}|x-G^{-1}h(t)|^{-d+\epsilon}{\d}x\Big){\d}t.
\end{equation*}
 Let $B(G^{-1}h(t),\rho):=\{x\in \mathbb{R}^{d}: |x-G^{-1}h(t)|\leq \rho\}$ satisfy that $|B(G^{-1}h(t),\rho)|=|\Omega|$,
which implies $|\Omega|=|\mathbb{S}^{d-1}|\rho^d$. By \eqref{eq:2233} with $A:=\Omega$, we obtain
\begin{equation*}
 \int_I\int_\Omega |Gx-h(t)|^{-d+\epsilon} {\d}x{\d}t
\leq\|G^{-1}\|^{d-\epsilon}\int_{0}^{T}\Big(\int_{B(G^{-1}h(t),\rho)}|x-G^{-1}h(t)|^{-d+\epsilon}{\d}x\Big){\d}t.
\end{equation*}
Then changing to polar coordinates for the inner integral leads to
\begin{align}
\int_I\int_\Omega |Gx-h(t)|^{-d+\epsilon} \d x\d t&\leq \|G^{-1}\|^{d-\epsilon} T\int_{0}^\rho r^{-d+\epsilon} r^{d-1}\mathrm{d}r |\mathbb{S}^{d-1}|\nonumber\\
&=\epsilon^{-1}\|G^{-1}\|^{d-\epsilon} T|\Omega|^{\frac{\epsilon}{d}}|\mathbb{S}^{d-1}|^{1-\frac{\epsilon}{d}}.\label{eq:777}
\end{align}
This proves $f(x,t)\in L^2(I;L^2(\Omega))$ for all $r>-\frac{d}{2}$. Next, the derivative $\dot f$
of $f=|Gx-h(t)|^r$ is given by $\dot f= -r|Gx-h(t)|^{r-2}(Gx-h(t))^{t}\dot{h}(t).$ Thus,
$|\dot f|\leq r|Gx-h(t)|^{r-1}|\dot{h}(t)|,$ which, together with \eqref{eq:777},
implies that for $r=-\frac{d}{2}+1+\frac{\epsilon}{2}$, there holds
\[
\|\dot f\|_{L^2(I;L^2(\Omega))}\leq r\sqrt{\epsilon^{-1}\|G^{-1}\|^{d-\epsilon} T|\Omega|^\frac{\epsilon}{d}|\mathbb{S}^{d-1}|^{1-\frac{\epsilon}{d}}} \|\dot{h}\|_{L^2(I)}.
\]
Thus, $\|f\|_{H^1(I;L^2(\Omega))}<\infty.$ By interpolation, for any $-\frac{d}{2}+\frac{\epsilon}{2}<r<-\frac{d}{2}+1+\frac{\epsilon}{2}$,
there holds $|Gx-h(t)|^r\in H^r(I;L^2(\Omega))$, with $s=r+\frac{d}{2}$. The general case of any $r>-\frac{d}{2}+1$ can
be analyzed analogously. The desired assertion follows since the constant $\epsilon\in(0,\frac12)$ can be made arbitrarily small.
\end{proof}

Lemma \ref{lem:reg1} does not cover the case $r=-\frac d2$,
which is treated next.
\begin{lemma}\label{lem:reg3}
Let Assumption \ref{ass:h} hold. Then for any $p\in(1,2)$, the function $f(x,t)=|Gx-h(t)|^{-\frac d2} \in
L^{2-\epsilon}(\Omega;L^p(I))$ for any $\epsilon>0$.
\end{lemma}
\begin{proof}
We only need to show
\begin{align}\label{eq:1122}
g(x):=\int_I|Gx-h(t)|^{-\frac{pd}{2}}\mathrm{d}t\in L^{\frac{2}{p},\infty}(\Omega) \quad\text{ for all } p\in (1,2).
\end{align}
Note that $L^{2,\infty}(\Omega)\subset L^{2-\epsilon}(\Omega)$ for all $\epsilon>0$ \cite[Section 6.4]{Folland:1999}, and hence,
$f(x,t)\in L^{2-\epsilon}(\Omega;L^p(I))$ for any $p\in(1,2)$. Next we prove \eqref{eq:1122}. Since
${L^{\frac{2}{p},\infty}(\Omega)}$ is a Banach space with a well-defined norm, we deduce
\begin{align}\label{eq:1234}
\|g\|_{L^{\frac{2}{p},\infty}(\Omega)}=\Big\|\int_I|Gx-h(t)|^{-\frac{pd}{2}}\mathrm{d}t\Big\|_{L^{\frac{2}{p},\infty}(\Omega)}\leq \int_I\Big\||Gx-h(t)|^{-\frac{pd}{2}}\Big\|_{L^{\frac{2}{p},\infty}(\Omega)}\mathrm{d}t.
\end{align}
Let $w(x,t):=|Gx-h(t)|^{-\frac{pd}{2}}$. To estimate $\|g\|_{L^{\frac{2}{p},\infty}
(\Omega)}$, we first compute the nonincreasing rearrangement $w^{*}(\tau,t)$ for
any fixed $t\in I$, which, by definition, is defined for all $\tau\geq 0$ by
\begin{align*}
w^*(\tau,t)&=\inf\{c>0: |\{x\in \Omega: w(x,t)>c\}|\leq\tau\}\\
&=\inf\{c>0: |\{x\in \Omega: |Gx-h(t)|<c^{-\frac{2}{pd}}\}|\leq\tau\}.
\end{align*}
This and the inclusion relation $\{x\in \Omega: |Gx-h(t)|<c^{-\frac{2}{pd}}\}\subset \{x\in \Omega:
|x-G^{-1}h(t)|<\|G^{-1}\|c^{-\frac{2}{pd}}\}$ (due to the trivial inequality $|Gx-h(t)|\geq \|G^{-1}\|^{-1}|x-G^{-1}h(t)|$) yield
\begin{align*}
w^*(\tau,t)\leq \Big(|\mathbb{S}^{d-1}|\|G^{-1}\|^d \tau^{-1} \Big)^{\frac{p}{2}}.
\end{align*}
Hence, for any fixed $t\in I$, we obtain
\begin{align*}
\Big\||Gx-h(t)|^{-\frac{pd}{2}}\Big\|_{L^{\frac{2}{p},\infty}(\Omega)}&=\|w(x,t)\|_{L^{\frac{2}{p},\infty}(\Omega)}\leq \sup\limits_{\tau\geq 0} \tau ^{\frac{p}{2}}\Big(|\mathbb{S}^{d-1}|\|G^{-1}\|^d \tau^{-1} \Big)^{\frac{p}{2}}\\
&=\Big(|\mathbb{S}^{d-1}|\|G^{-1}\|^d \Big)^{\frac{p}{2}},
\end{align*}
which, in view of \eqref{eq:1234}, yields \eqref{eq:1122}. This completes the proof of the lemma.
\end{proof}
Now we can state the degree of ill-posedness for the limit problem for FFP trajectories.
\begin{theorem}\label{thm:limit1}
For $\beta \rightarrow \infty$, if $p\in (L^\infty(\Omega))^d$, $h\in (H^s(I))^d$ with $s>\frac32$,
$d=2,3$, for FFP trajectories, the SVs $\sigma_n$ of the operator
$\widetilde F$ defined in \eqref{eqn:problem-limit-case} decay as {\rm(}for any $\epsilon\in(0,\frac14)${\rm)}:
\begin{equation*}
\sigma_n \leq \left\{\begin{array}{ll}
  Cn^{-1+\epsilon}, & d=3,\\
  Cn^{-\frac{1}{4}+\epsilon}, & d=2.
  \end{array}\right.
\end{equation*}
\end{theorem}
\begin{proof}
By the Cauchy-Schwarz inequality, there holds $|\tilde{\kappa}| \leq \mu_0m_0 |p||H|^{-1}|\dot{H}|.$
Now we discuss the cases $d=2$ and $d=3$ separately. For $d=3$, by Lemma \ref{lem:reg1},
$|H|^{-1}\in H^{\frac{1}{2}-\epsilon}(I;L^2(\Omega))$, and since $s>\frac32$, by
Sobolev embedding, $\dot H = - \dot h\in (H^{s-1}(I))^3\subset (L^\infty(I))^3$.
By Theorem \ref{thm:product} and $p\in (L^\infty(\Omega))^3$,
we have $| H|^{-1} |\dot{h}|\in H^{\frac{1}{2}-\epsilon}(I;L^2(\Omega))$. Then the assertion
follows from Theorem \ref{thm:svd-decay-reg}. The case $d=2$ is similar: since $
|\dot{h}|\in L^\infty(I)$ and $|p|\in L^\infty(\Omega)$,  we apply Theorem
\ref{thm:svd-decay} to obtain the desired assertion.
\end{proof}

\begin{remark}
Theorem \ref{thm:limit1} indicates that the SVs of the operator
$\widetilde{F}$ decay faster as the spatial dimension $d$ increases from $2$ to $3$ {\rm(}in terms of
the upper bounds{\rm)}. Provided that a certain minimal assumption on $h$ is satisfied,
this result suggests increasing $\beta$ to improve the resolution.
\end{remark}

In practice, one can also have FFL trajectories for $d=3$, where $G\in\mathbb{R}^{3\times 3}$ has
only rank $2$. Then for any fixed $t\in I$, the kernel function $\tilde\kappa(x,t)$ is singular along a line
in $\Omega$, instead of at one single point. We analyze a simplified model to gain insight. Since $\mathrm{rank}(G)
=2$ and symmetric, by properly changing the coordinate, we may assume that $G$ is diagonal with
the last diagonal entry being zero. Then the condition for any fixed $t\in I$, there exists $x\in \Omega$
such that $Gx=h(t)$ implies $h_3(t)=0$, and the singular kernel essentially depends only on $h_1(t)$ and
$h_2(t)$, and the third component of $H(x,t)$ vanishes. Thus for a cylindrical domain $\Omega=\Omega_{12}\times \Omega_3$,
the forward operator $\tilde{F}$ can be reformulated as in the 2D FFP trajectories (with respect to the
average  $\int_{\Omega_3}c(x){\rm d}x_3$ of the concentration $c(x)$):
\begin{align*}
v(t) = - \mu_0m_0\int_{\Omega_{12}} \Big(\int_{\Omega_3}c(x)\d x_3\Big)p(x)^t\left(-\frac{H H^t}{|H|^3} + \frac{1}{|H|} I_3 \right) \dot{H}\d x_1\d x_2.
\end{align*}

The next result analyzes the degree of ill-posedness of FFL trajectories under the designate conditions.
\begin{theorem}\label{thm:limit2}
Under the preceding assumptions, for $\beta \rightarrow \infty$, for FFL trajectories in 3D, $p\in (L^\infty
(\Omega))^3$, and $h(t)\in (H^s(I))^3$ with $s>\frac32$, the SVs $\sigma_n$ of the operator $\widetilde F$
decay as $\sigma_n\le Cn^{-\frac14+\epsilon}$, for any $\epsilon\in(0,\frac14)$.
\end{theorem}
\begin{proof}
The proof is similar to Theorem \ref{thm:limit1} with
$d=2$. Since $s>\frac32$, by Sobolev embedding, $\dot H = -\dot h(t)\in (H^{s-1}(I))^3
\hookrightarrow (L^\infty(I))^3$. Since the third component of $H$ vanishes, this and Theorem
\ref{thm:svd-decay} yield the assertion.
\end{proof}

\begin{remark}
\label{rem:conjecture2}
In Theorems \ref{thm:limit1} and \ref{thm:limit2}, the trajectory $h(t)$ is assumed to be $(H^s(I))^d$,
with $s>\frac32$. This restriction comes from the requirement $\dot h(t)\in (L^\infty(I))^d$ to
simplify the analysis. The estimates in Theorem \ref{thm:limit2} and Theorem \ref{thm:limit1} for
$d=2$ are conservative, due to suboptimal bound in Theorem \ref{thm:svd-decay} {\rm(}see Remark \ref{rem:conjecture}{\rm)}.
\end{remark}

\subsection{Filtered model}\label{ssec:filter}
In practice, the signal is first preprocessed by an analog filter to remove
the excitation so that the true signal is not lost during digitalization.
This can be achieved by a band stop filter. Mathematically, it amounts to
convolution with a given kernel $a(t):\bar I := [-T,T]\to\mathbb{R}$ defined by
\begin{equation}\label{eqn:kernel-hat}
\widehat \kappa(x,t) = \int_I \kappa(x,t') a(t-t')\d t'\quad \forall t\in I.
\end{equation}
However, the precise form of the filter $a(t)$ remains elusive, which currently constitutes one of the major challenges
in realistic mathematical modeling of MPI \cite{Kluth:2017}. To analyze the influence of the filtering step,
we recall a smoothing property of the convolution operator
\cite[Theorem 3]{Burenkov:1989}. Below the notation $(\cdot)_+$ denotes the positive part. We refer to
\cite{Triebel:1978} for a treatise on Besov spaces $B_{p,\theta}^s(\mathbb{R})$.
\begin{lemma}\label{lem:conv}
For $-\infty<\ell_1,\ell_2<\infty$, $1\leq p_1\leq p_2 \leq \infty$, $0<\theta_1,\theta_2\leq \infty$,
with $\frac{1}{p} = \frac{1}{p_1'} + \frac{1}{p_2}$ and $\frac{1}{\theta}\geq (\frac{1}{\theta_2}-\frac{1}{\theta_1})_+$, then for $f\in B_{p,\theta}^{
\ell_2-\ell_1}(\mathbb{R})$ and $g\in B_{p_1,\theta_1}^{\ell_1}(\mathbb{R})$, the convolution $f\ast g$
exists and
\begin{equation*}
\|f\ast g\|_{B_{p_2,\theta_2}^{\ell_2}(\mathbb{R})} \leq C\|f\|_{B_{p,\theta}^{\ell_2-\ell_1}(\mathbb{R})}
\|g\|_{B_{p_1,\theta_1}^{\ell_1}(\mathbb{R})}.
\end{equation*}
\end{lemma}

Last, we describe the influence of filtering: the SVs $\sigma_n$ of the filtered model decay
faster than the nonfiltered one by a factor $r$, the regularity index of the filter.
\begin{theorem}\label{thm:filter}
Suppose that the zero extension of the filter $a(t):\bar I\to\mathbb{R}$ belongs to $B_{1,\theta}^r(\mathbb{R})$,
for some $r\geq0$ and $0<\theta\leq \infty$, and the conditions in Theorem \ref{thm:non-filter} hold. Then the
SVs $\sigma_n$ of the operator $\widehat F$ for the kernel $\widehat \kappa(x,t)$
defined in \eqref{eqn:kernel-hat} decay as $\sigma_n\leq Cn^{\frac12-s-r}$.
\end{theorem}
\begin{proof}
Let $\bar \kappa$ be any bounded extension of $\kappa$ from $H^{s-1}(I;L^2(\Omega))$ to $H^{s-1}(\mathbb{R};L^2(\Omega))$,
and denote by $\bar a$ the zero extension of $a:\bar I\to \mathbb{R}$ to $\mathbb{R}\setminus\bar{I}$. Then we can extend $\widehat \kappa$
from $I$ to $\mathbb{R}$, still denoted by $\widehat \kappa$, by $\widehat\kappa(x,t)=\int_\mathbb{R} \bar \kappa(x,t')
\bar a(t-t')\d t'.$ Clearly, the restriction of $\widehat\kappa(x,t)$ to $I$ coincides with $\widehat \kappa$ defined
in \eqref{eqn:kernel-hat}, by the construction of the extension $\bar a$. This and the fact $H^{s}(\mathbb{R})=B_{2,2}^s(\mathbb{R})$
\cite[Remark 4, p. 179]{Triebel:1978} imply
\begin{align*}
  \|\widehat\kappa\|_{H^{s+r-1}(I;L^2(\Omega))}&=\|\bar\kappa\ast \bar{a}\|_{H^{s+r-1}(I;L^2(\Omega))} \leq \|\bar \kappa\ast \bar{a}\|_{H^{s+r-1}(\mathbb{R};L^2(\Omega))}\\
   &\leq C\|\bar\kappa\|_{H^{s-1}(\mathbb{R};L^2(\Omega))}\|\bar a\|_{B_{1,\theta}^r(\mathbb{R})}\leq
   C\|\kappa\|_{H^{s-1}(I;L^2(\Omega))}\|\bar a\|_{B_{1,\theta}^r(\mathbb{R})},
\end{align*}
where the second inequality is due to Lemma \ref{lem:conv} and the last one due to the bounded extension.
Then the assertion follows from Theorem \ref{thm:svd-decay-reg}.
\end{proof}

\section{Numerical results}\label{sec:numer}
Now we illustrate the theoretical results with numerical examples for the non-filtered
FFP and FFL cases. We do not study the influence of analog filter, since its mechanism and
precise form are still poorly understood.

\subsection{Setting of numerical experiments}
In our numerical simulation, we use parameters that are comparable with real experiments. We parameterize
problems \eqref{eqn:non-filtered} and \eqref{eqn:problem-limit-case} analogously to the Bruker FFP scanner, and
obtain the parameter values from a public dataset \cite{knopp2016mdf}. Sinusoidal excitation patterns are
used to move the FFP along Lissajous trajectories, which are often employed in practice due to its fast
coverage of the domain of interest. In the FFP case, the gradient field $g:\R^3 \rightarrow \R^3$
is taken to be linear, i.e., $g(x)= Gx$ with $G\in \R^{3\times 3}$ diagonal and $\mathrm{trace}(G)=0$.
The drive field $h:I\rightarrow \R^3$ is taken to be trigonometric, i.e. $h(t)=-(A_1 \sin
(2\pi f_1 t),A_2 \sin(2\pi f_2 t),A_3 \sin(2\pi f_3 t))^t$, $A_i,f_i>0$, $i=1,2,3$. By Theorem
\ref{thm:non-filter2}, for $0<\beta<\infty$, the good spatial regularity of the kernel $\kappa(x,t)$
precludes examining the influence of trajectory smoothness on the SV decay. Nonetheless,
we consider also triangular trajectories, i.e., $h(t)=-(A_1 \mathrm{tri}(2\pi f_1 t),A_2 \mathrm{tri}
(2\pi f_2 t),A_3 \mathrm{tri}(2\pi f_3 t))^t$, $A_i,f_i>0$, $i=1,2,3$, where the function
$\mathrm{tri}:\R \rightarrow [-1,1]$ is defined by
\begin{equation*}
 \mathrm{tri}(z)=\begin{cases}
                  \frac{2}{\pi} \theta & \theta=(z \textrm{ mod } 2\pi) \in [0, \pi/2),\\
                  2- \frac{2}{\pi} \theta & \theta=(z \textrm{ mod } 2\pi) \in [\pi/2, 3\pi/2),\\
                  -4 +\frac{2}{\pi} \theta & \theta=(z \textrm{ mod } 2\pi) \in [3\pi/2, \pi/2).\\
                 \end{cases}
\end{equation*}
We refer to Fig. \ref{fig:Lissajous2D} for an illustration of Lissajous excitation with sinusoidal and triangular trajectories.

\begin{figure}[hbt!]
\centering
 	\begin{subfigure}[t]{0.35\textwidth}
 		\includegraphics[width=\textwidth,trim={1cm 0 4.2cm 1cm},clip]{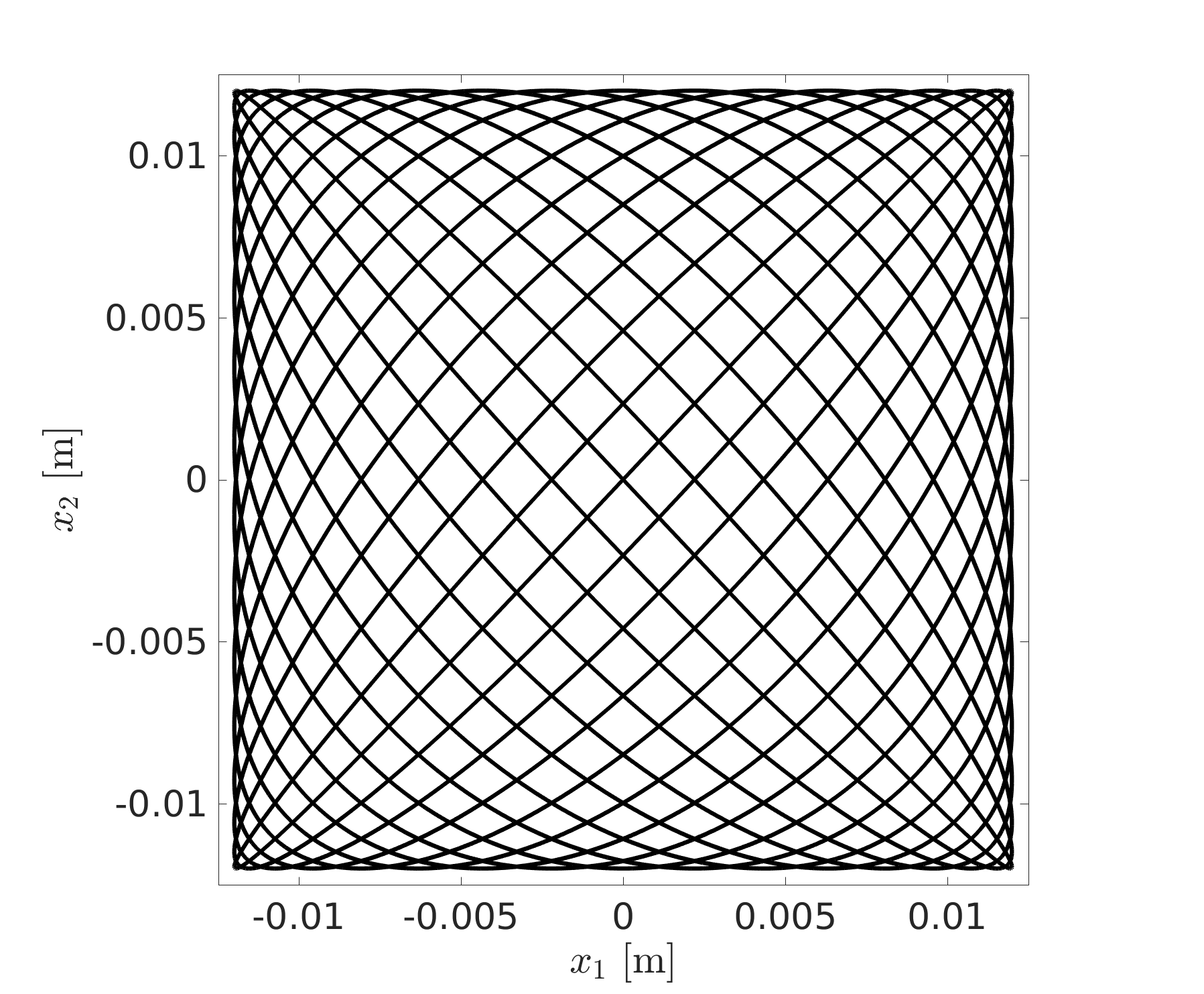}
		\caption{sinusoidal}
	\end{subfigure}
    \begin{subfigure}[t]{0.35\textwidth}
 		\includegraphics[width=\textwidth,trim={1cm 0 4.2cm 1cm},clip]{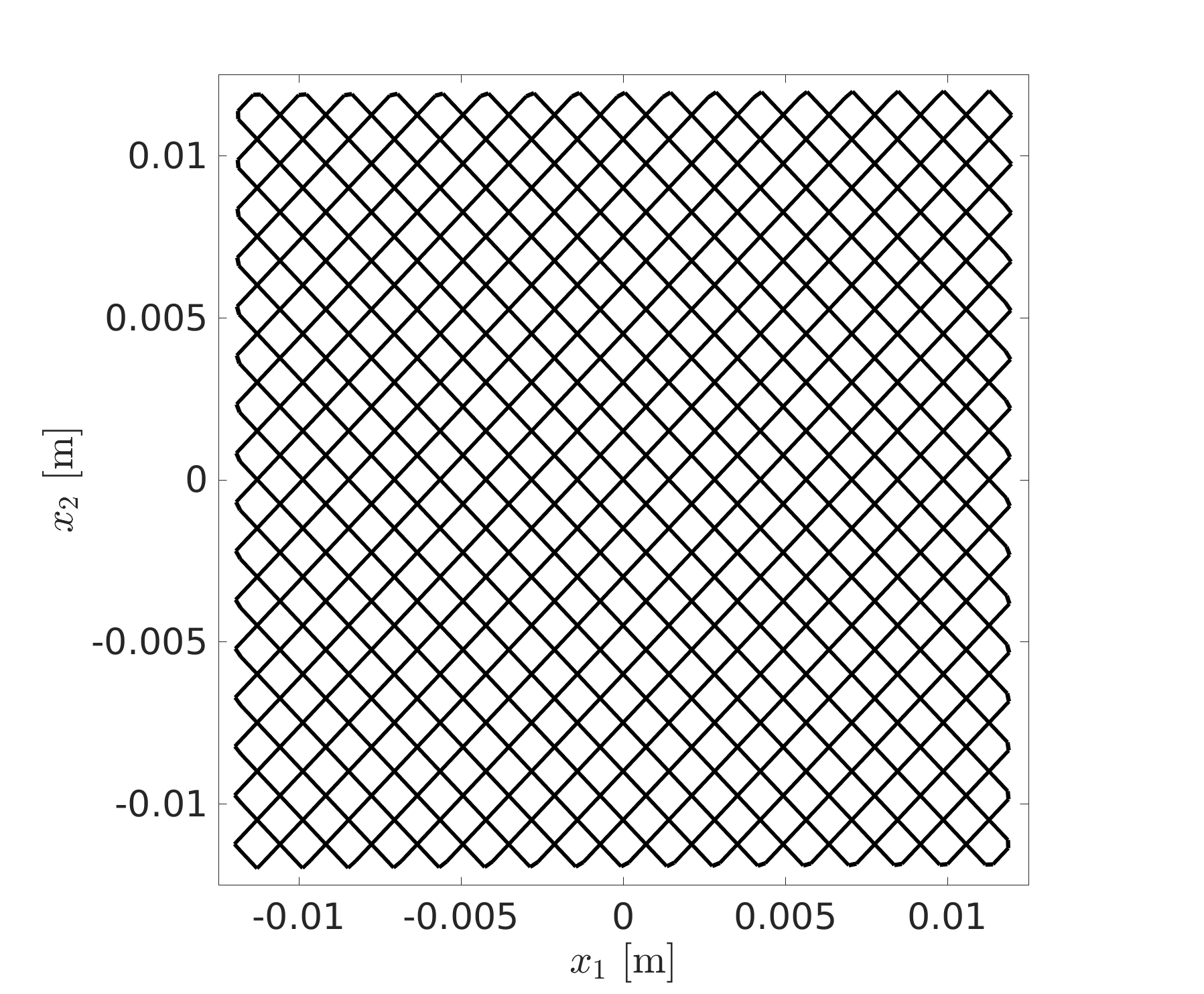}
		\caption{triangular}
	\end{subfigure}	
 \caption{FFP trajecories Lissajous curve for sinusoidal (a) and triangular (b) excitation in 2D.}
 \label{fig:Lissajous2D}
\end{figure}

The scanner is equipped with three receive coil units, each sensitive to one component of the 
mean magnetic moment vector $\bar m(x,t)$. The coil sensitivities $p$ are assumed to be
homogeneous in the simulation. 
The particle's magnetic moment $m_0$
is set to $m_0=\mu_0^{-1}$ in order to avoid diameter-dependent numerical errors, and the remaining
particle parameters are chosen following \cite{Deissler2014}.

The forward operators $F$, $\tilde{F}: L^2(\Omega) \rightarrow (L^2(I))^L$ are discretized by a
Galerkin method with piecewise constant basis functions on a rectangular partition of the spatial
domain $\Omega$ and  uniform partition of the time interval $I$. The integrals are computed using
a quasi Monte-Carlo quadrature rule in the spatial variable $x$, with a Halton sequence with $3^d$
points, for which a good accuracy was reported  \cite{morokoff1995quasi}. It can also approximate
singular integrals \cite{owen2006quasi} arising in the limit case. The time integral is computed
by a Gaussian quadrature rule, and interval splitting is performed if necessary such that discontinuities
lie at the interval boundaries only. MATLAB's built-in function \texttt{svd} is used to compute SVs of the discrete
model. Note that in practice, the computed small SVs are polluted by discretization
errors and quadrature errors etc, and thus the discussions below focus on leading SVs, although the
theoretical results hold for all $n$.

\begin{table}[hbt!]
\centering
\caption{Physical parameters used for the simulations. The parameters can be found in: FFP scanner
setup \cite{knopp2016mdf} and particle parameters \cite{Deissler2014}.}
\label{tab:sim-parameters}
{\scriptsize
 \begin{tabular}{ll|l|l|l}
 \midrule
 \textbf{Parameter} & & \multicolumn{3}{l}{\textbf{Value}}\\
 \midrule
 Magnetic permeability & $\mu_0$ & \multicolumn{3}{l}{$4\pi\times 10^{-7} \text{ H/m}$}\\
 Boltzmann constant & $k_\mathrm{B}$ & \multicolumn{3}{l}{$1.38064852\times 10^{-23} \text{ J/K}$}\\

\midrule
\multicolumn{5}{l}{{\it Particle}} \\
\midrule
  Temperature & $T_\textrm{B}$ & \multicolumn{3}{l}{$293 \text{ K}$} \\
  Sat. magnetization & $M_\mathrm{S}$ & \multicolumn{3}{l}{$474000 \text{ J/m$^3$/T}$}\\
  Particle core diameter & $D$ & \multicolumn{3}{l}{$\in \{20,30,40\} \times 10^{-9} \text{ m}$}  \\
  Particle core volume & $V_\mathrm{C}$ & \multicolumn{3}{l}{$1/6 \pi D^3$} \\
   & $m_0$ & \multicolumn{3}{l}{$\frac{1}{\mu_0}$} \\
   & $\beta$ & \multicolumn{3}{l}{$\frac{\mu_0V_\mathrm{C}M_\mathrm{S}}{k_\mathrm{B}T_\textrm{B}}$} \\
\midrule
\multicolumn{1}{l}{{\it Geometry }} & $d$ & 1 & 2 & 3\\
\midrule
FOV & $e_1$ & $[-12.5,12.5]  \text{ mm}$ & $[-12.5,12.5]  \text{ mm}$& $[-12.5,12.5]  \text{ mm}$\\
 & $e_2$ & $-$ & $[-12.5,12.5]  \text{ mm}$& $[-12.5,12.5]  \text{ mm}$\\
 & $e_3$ & $-$ & $-$& $[-6.5,6.5]  \text{ mm}$\\
Cuboid size & $\Delta x$ & $0.1\text{ mm}$ & $0.5 \times 0.5 \text{ mm}^2$& $1.0 \times 1.0 \times 1.0  \text{ mm}^3$\\
\midrule
\multicolumn{5}{l}{{\it Scanner FFP case }}  \\
\midrule
 Excitation frequencies & $f_1$& $2.5/102 \times 10^6 \text{ Hz}$ & $2.5/102 \times 10^6 \text{ Hz}$ & $2.5/102 \times 10^6 \text{ Hz}$\\
  & $f_2$& $-$ & $2.5/96 \times 10^6 \text{ Hz}$ & $2.5/96 \times 10^6 \text{ Hz}$\\
 & $f_3$& $-$ & $-$ & $2.5/99 \times 10^6 \text{ Hz}$\\
 Excitation amplitudes & $A_1$& $0.012\text{ T}/\mu_0$ & $0.012\text{ T}/\mu_0$& $0.012\text{ T}/\mu_0$\\
  & $A_2$& $-$ &$0.012\text{ T}/\mu_0$ & $0.012\text{ T}/\mu_0$\\
  & $A_3$& $-$ & $-$ & $0.012\text{ T}/\mu_0$\\

 Gradient strength & $G_{1,1}$ & $-1\text{ T/m}/\mu_0$ & $-1\text{ T/m}/\mu_0$ & $-1\text{ T/m}/\mu_0$\\
 & $G_{2,2}$ & $-1\text{ T/m}/\mu_0$ & $-1\text{ T/m}/\mu_0$ & $-1\text{ T/m}/\mu_0$\\
 & $G_{3,3}$ & $2\text{ T/m}/\mu_0$ & $2\text{ T/m}/\mu_0$ & $2\text{ T/m}/\mu_0$\\
Measurement time & $T$ & $0.04 \times 10^{-3} \text{ s}$ & $0.653 \times 10^{-3} \text{ s}$ & $21.54 \times 10^{-3} \text{ s}$\\
  & $\Delta t$ &$0.01 \times 10^{-6} \text{ s}$ & $0.2 \times 10^{-6} \text{ s}$ & $0.4 \times 10^{-6} \text{ s}$  \\
\midrule
\multicolumn{5}{l}{{\it Scanner FFL case}}  \\
\midrule
 Excitation frequencies & $f_1$& $-$ & $-$ & $-$\\
  & $f_2$& $-$ & $2.5/96 \times 10^6 \text{ Hz}$ & $2.5/96 \times 10^6 \text{ Hz}$\\
 & $f_3$& $-$ & $-$ & $2.5/96/25/20 \times 10^6 \text{ Hz}$\\
 Field rotation & $f_\textrm{rot}$& $-$ & $2604.17 \text{ Hz}$ & $2604.17 \text{ Hz}$\\
 Excitation amplitudes & $A_1$& $-$ & $-$& $-$\\
  & $A_2$& $-$ & $0.012\text{ T}/\mu_0$ & $0.012\text{ T}/\mu_0$\\
  & $A_3$& $-$ & $-$ & $0.06\text{ T}/\mu_0$\\

 Gradient strength & $G_{1,1}$ & $-$ & $0\text{ T/m}/\mu_0$ & $0\text{ T/m}/\mu_0$\\
 & $G_{2,2}$ & $-$ & $-1\text{ T/m}/\mu_0$ & $-1\text{ T/m}/\mu_0$\\
 & $G_{3,3}$ & $-$ & $1\text{ T/m}/\mu_0$ & $1\text{ T/m}/\mu_0$\\
Measurement time & $T$ & $-$ & $0.77 \times 10^{-3} \text{ s}$ & $19.2 \times 10^{-3} \text{ s}$\\
& $\Delta t$ & $-$ & $0.2 \times 10^{-6} \text{ s}$ & $0.4 \times 10^{-6} \text{ s}$  \\
\midrule
\end{tabular}
}
\end{table}

\subsection{Numerical results and discussions}

First, we compare the SV decays for the FFP case using sinusoidal excitation patterns
moving the FFP along a Lissajous trajectory in the 2D and 3D cases. In all cases, for $d=1,2,3$,
the first $d$ receive coils were used to compute the SVs. The parameters used in the numerical
simulation are summarized in Table \ref{tab:sim-parameters}. In all the figures,
the whole range of the SVs of the discrete problem is presented, where the small SVs
are not reliable. The FFP results including
the limit cases (2D/3D) are presented in Fig. \ref{fig:sv_FFP_sin},
where for the purpose of comparison, the reference decay rates $O(n^{-\frac12})$ and $O(n^{-1})$
are also shown. It is observed that as the particle diameter $D$ increases, the
decay rate approaches the theoretical one from Theorem \ref{thm:limit1}, and for small $D$,
the decay is exponential. The numerical results for the 2D FFP limit case agree well with
the predictions from Theorem \ref{thm:limit1} and the conjecture in Remark \ref{rem:conjecture}.
 In the 3D FFP limit case, the decay of the leading SVs is slightly slower than the
theoretical rate in Theorem \ref{thm:limit1}. Similar observations hold for triangular
excitations in the drive field and when using one single receive coil only. These results can be
found in Figs. \ref{fig:sv_FFP_tri} and \ref{fig:sv_FFP_sin_single}.

\begin{figure}[hbt!]
\centering
 	\begin{subfigure}[t]{0.32\textwidth}
 		\includegraphics[width=\textwidth,trim={1.4cm 0 4.8cm 0},clip]{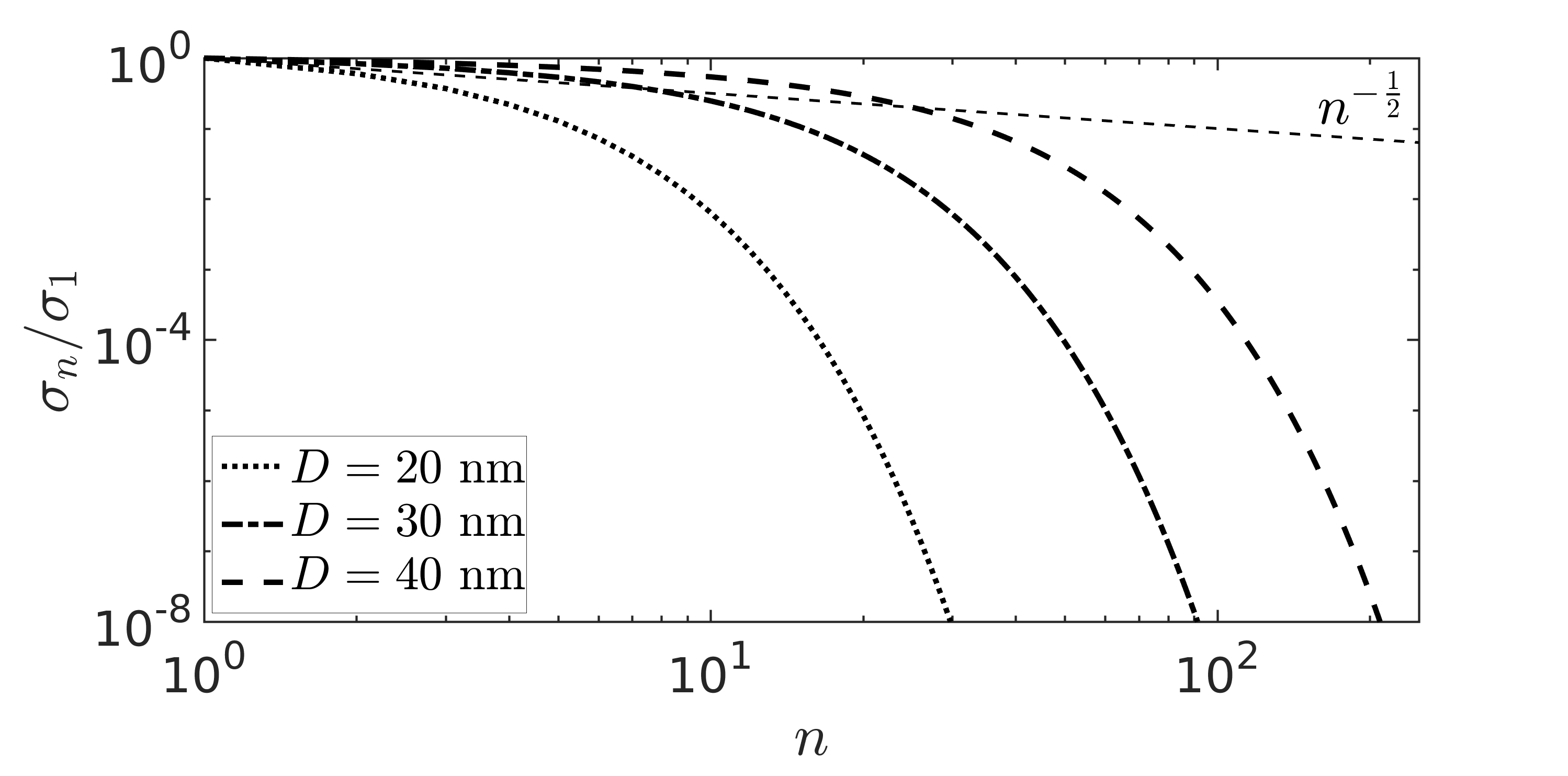}
		\caption{FFP 1D  }
		\label{subfig1:sv_FFP_sin}
	\end{subfigure}
    \begin{subfigure}[t]{0.32\textwidth}
 		\includegraphics[width=\textwidth,trim={1.4cm 0 4.8cm 0},clip]{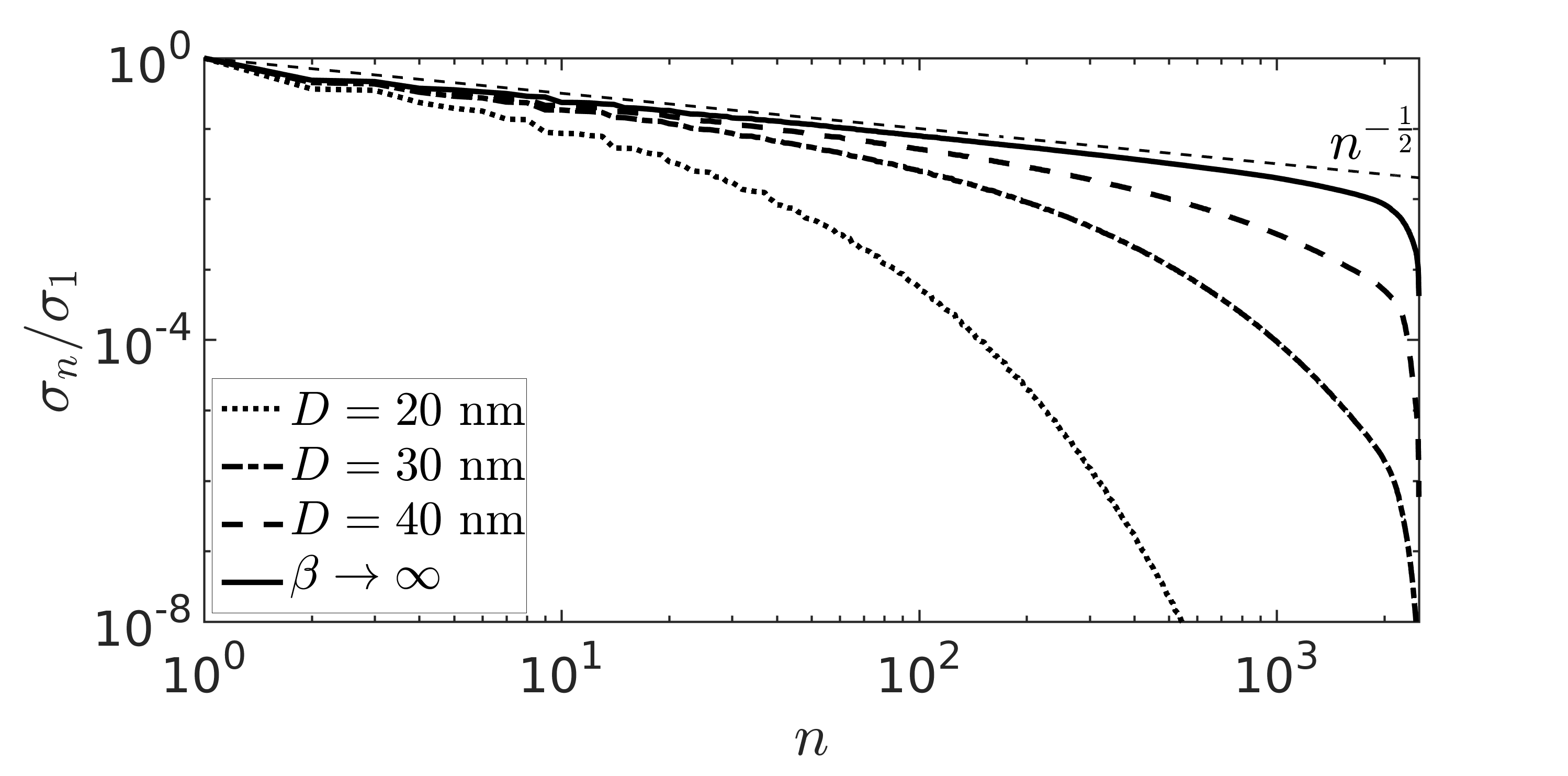}
		\caption{FFP 2D  }
		\label{subfig2:sv_FFP_sin}
	\end{subfigure}	
    \begin{subfigure}[t]{0.32\textwidth}
 		\includegraphics[width=\textwidth,trim={1.4cm 0 4.8cm 0},clip]{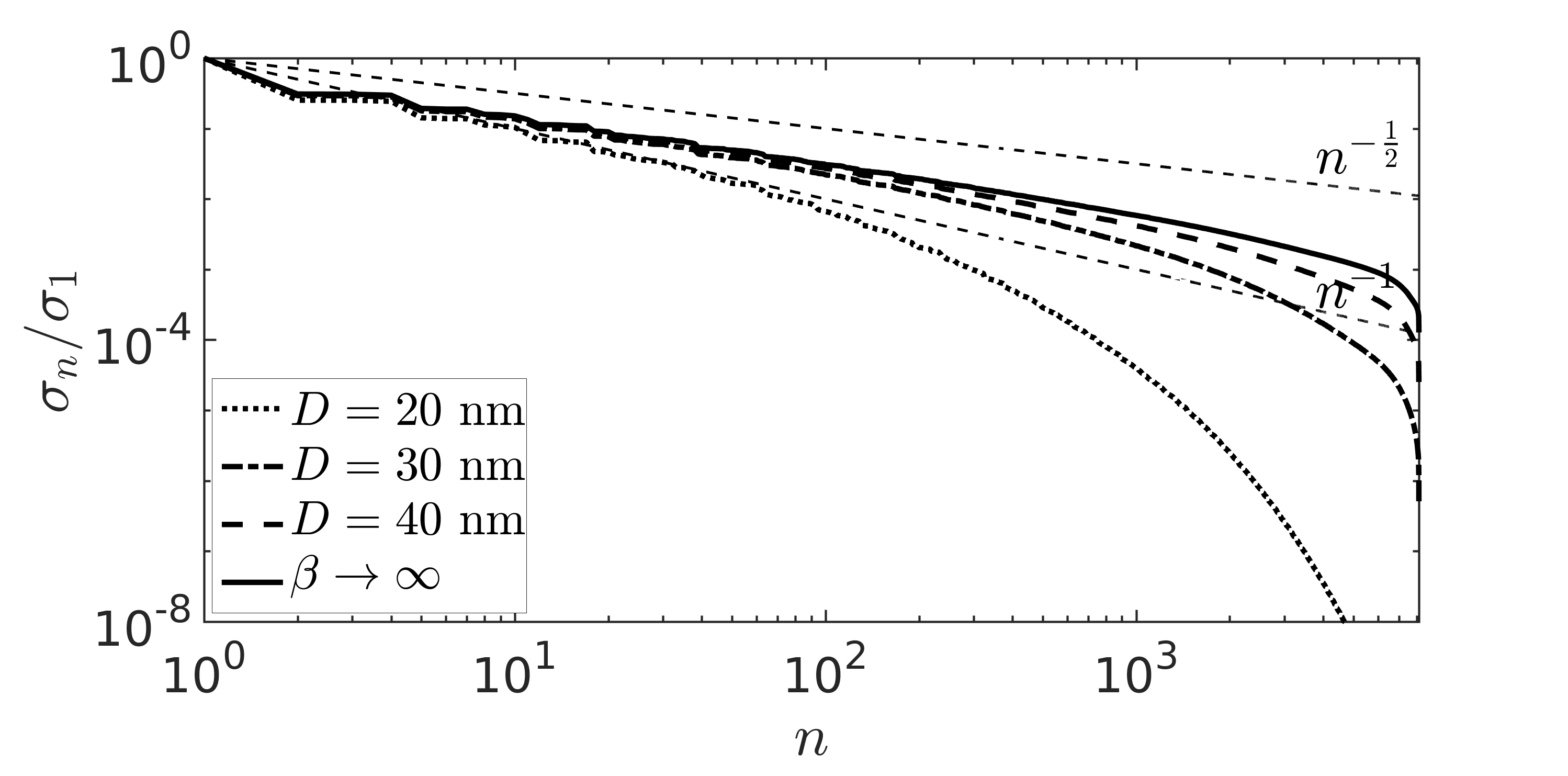}
		\caption{FFP 3D   }
		\label{subfig3:sv_FFP_sin}
	\end{subfigure}
 \caption{SV decay for sinusoidal excitation patterns and different particle diameters in (a) 1D, (b) 2D, and (c) 3D domains.}
 \label{fig:sv_FFP_sin}
\end{figure}

\begin{figure}[hbt!]
\centering
 	\begin{subfigure}[t]{0.32\textwidth}
 		\includegraphics[width=\textwidth,trim={1.4cm 0 4.8cm 0},clip]{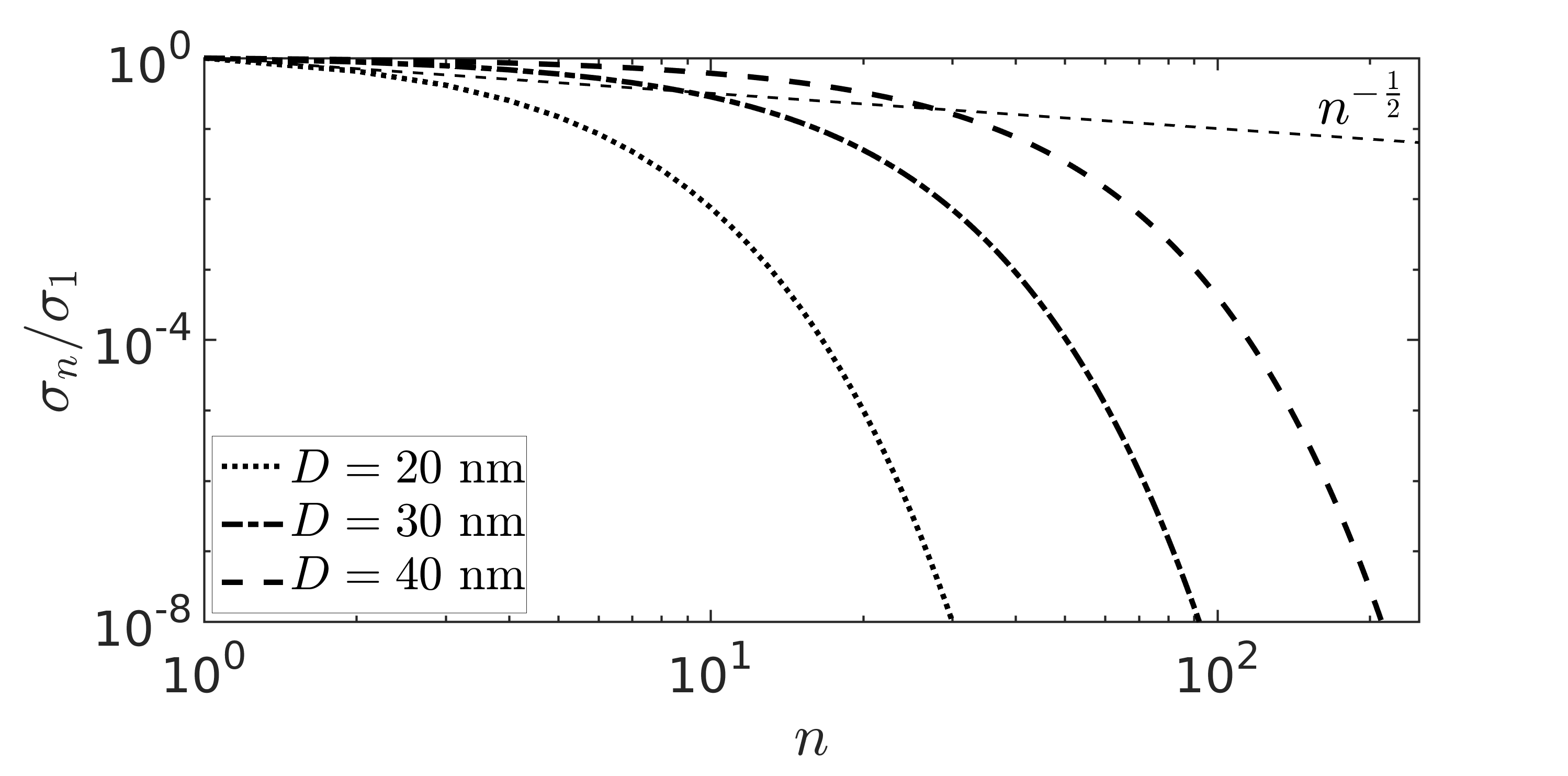}
		\caption{FFP 1D}
		\label{subfig1:sv_FFP_tri}
	\end{subfigure}
 	\begin{subfigure}[t]{0.32\textwidth}
 		\includegraphics[width=\textwidth,trim={1.4cm 0 4.8cm 0},clip]{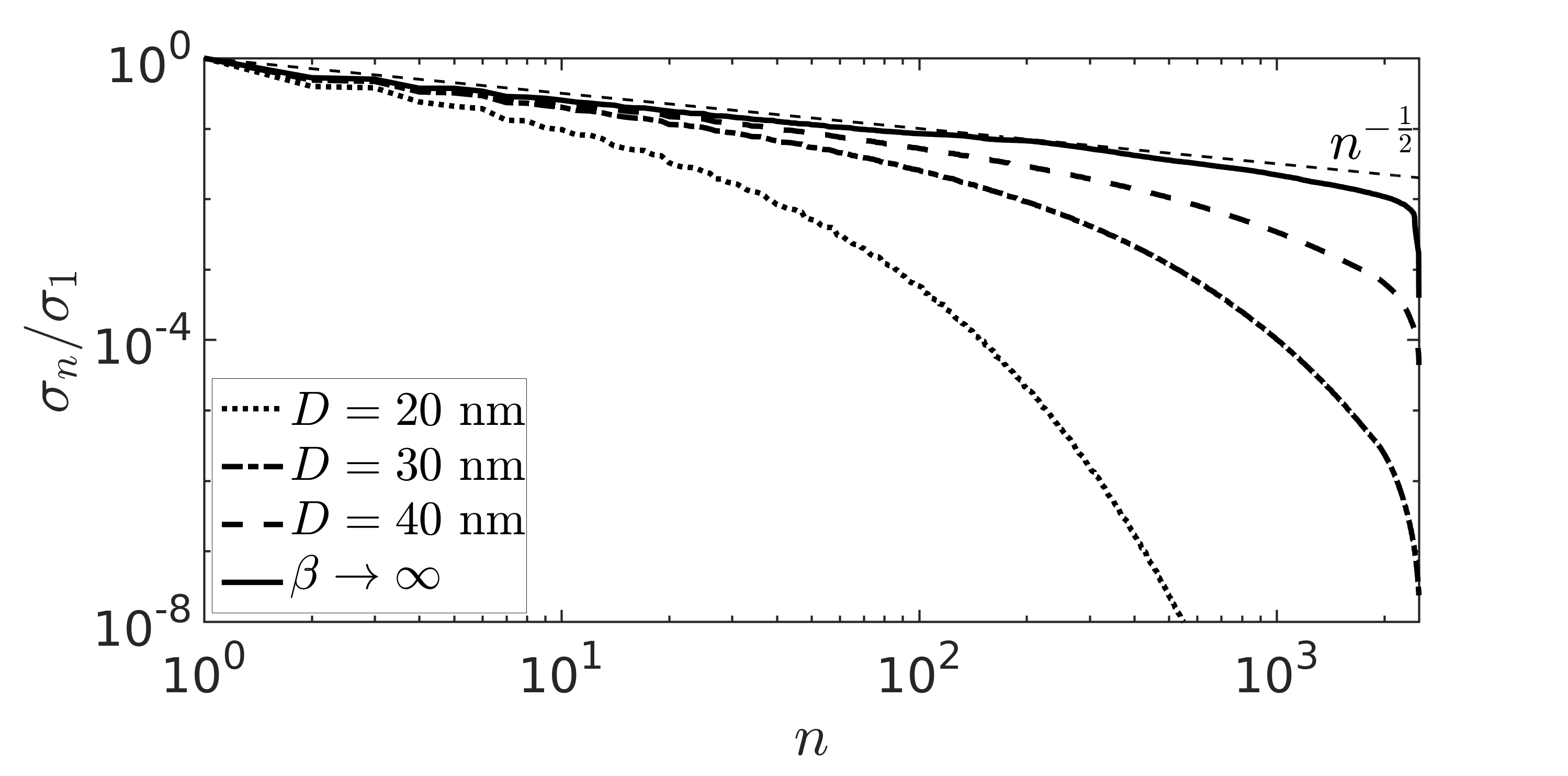}
		\caption{FFP 2D  }
		\label{subfig2:sv_FFP_tri}
	\end{subfigure}
	 	\begin{subfigure}[t]{0.32\textwidth}
 		\includegraphics[width=\textwidth,trim={1.4cm 0 4.8cm 0},clip]{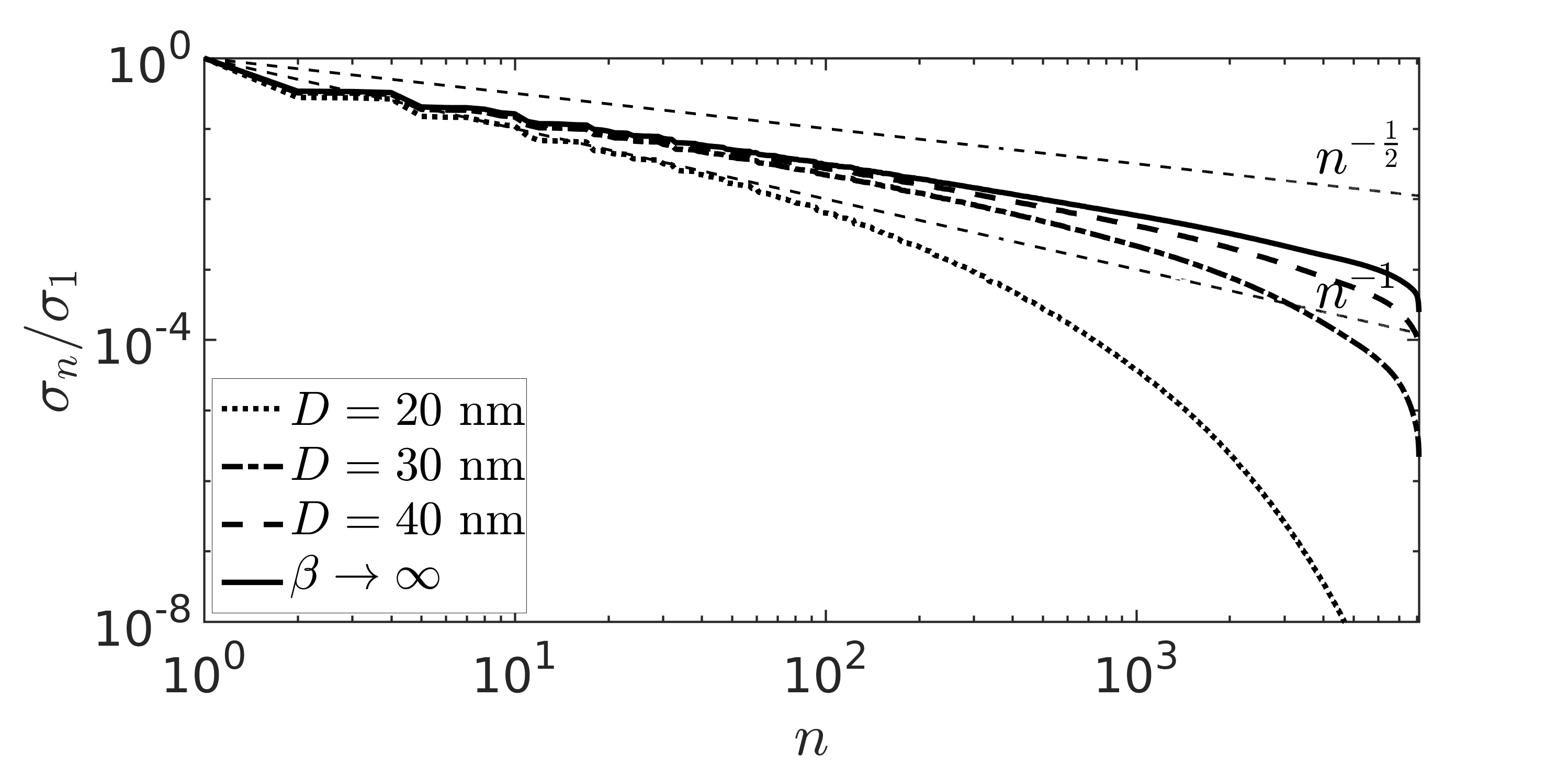}
		\caption{FFP 3D }
		\label{subfig3:sv_FFP_tri}
	\end{subfigure}
 \caption{SV decay for triangular excitation patterns and different particle diameters in (a) 1D, (b) 2D, and (c) 3D domains.
 }
 \label{fig:sv_FFP_tri}
\end{figure}

\begin{figure}
\centering
 	\begin{subfigure}[t]{0.32\textwidth}
 		\includegraphics[width=\textwidth,trim={1.4cm 0 4.8cm 0},clip]{figures/FFP_std_1D_series1_sv_decay_data_overview_1D_sin.png}
		\caption{FFP 1D  }
		\label{subfig1:sv_FFP_sin_single}
	\end{subfigure}
 	\begin{subfigure}[t]{0.32\textwidth}
 		\includegraphics[width=\textwidth,trim={1.4cm 0 4.8cm 0},clip]{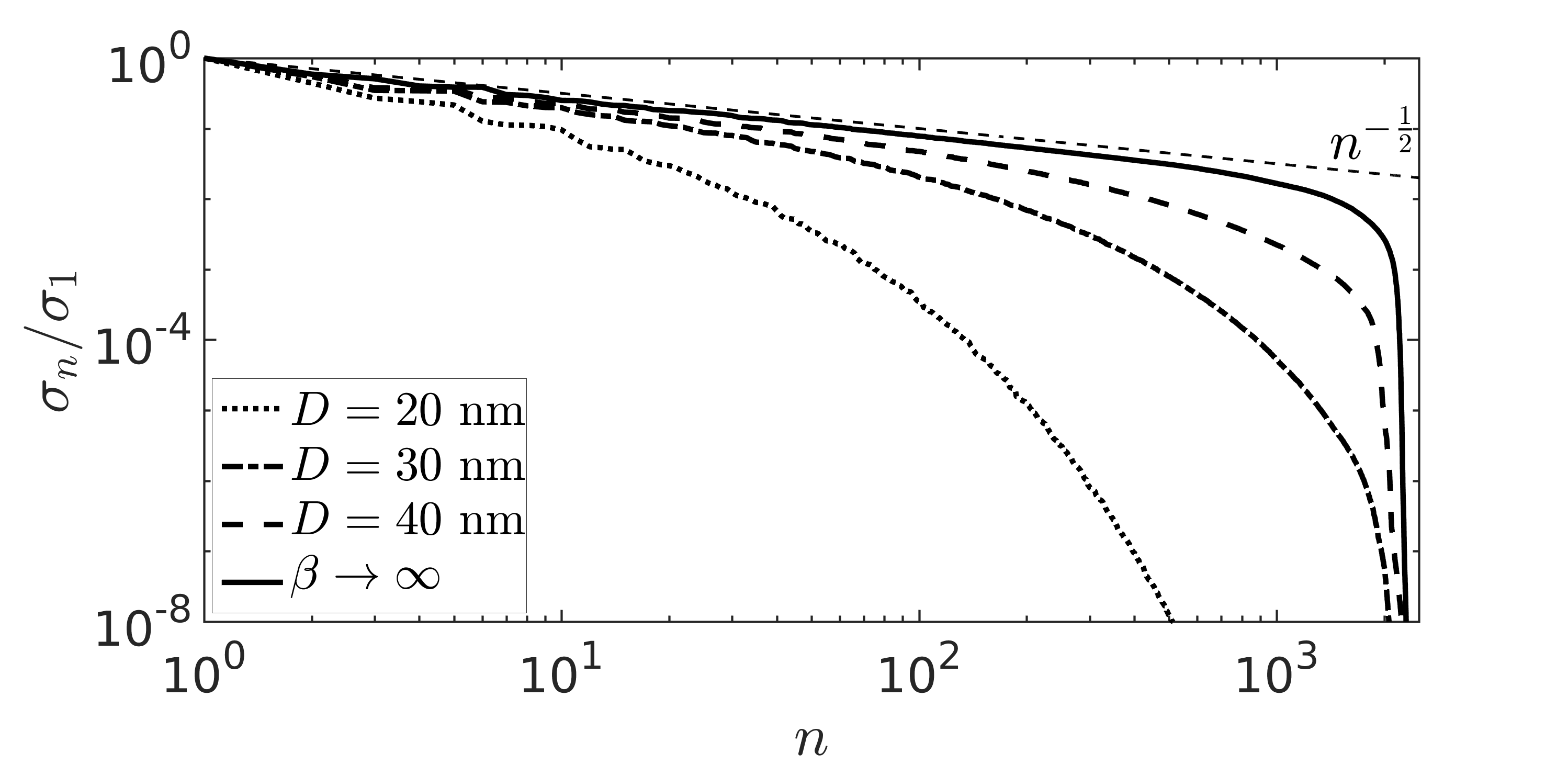}
		\caption{FFP 2D  }
		\label{subfig2:sv_FFP_sin_single}
	\end{subfigure}
	 	\begin{subfigure}[t]{0.32\textwidth}
 		\includegraphics[width=\textwidth,trim={1.4cm 0 4.8cm 0},clip]{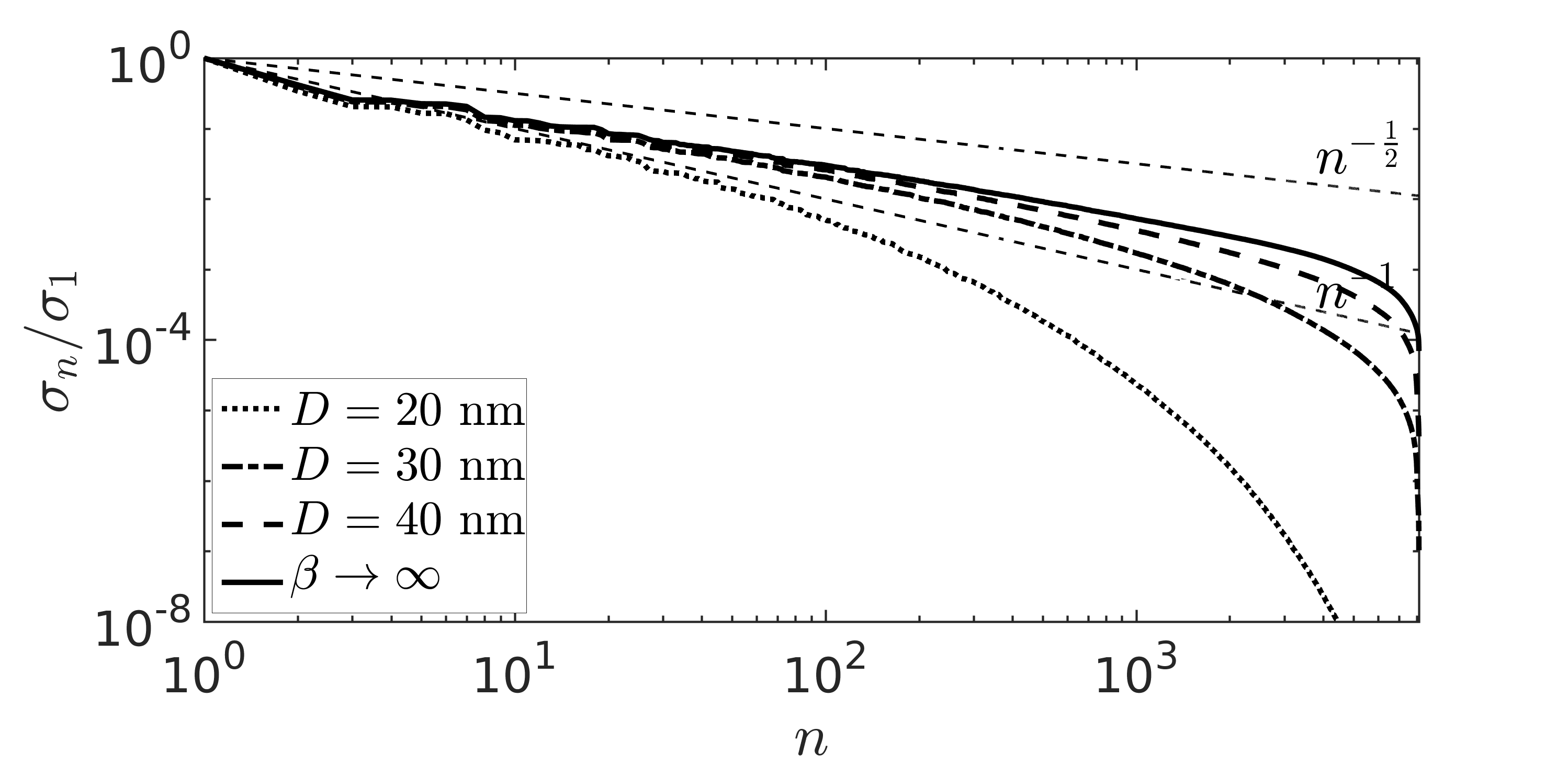}
		\caption{FFP 3D   }
		\label{subfig3:sv_FFP_sin_single}
	\end{subfigure}
 \caption{SV decay for sinusoidal excitation patterns and different particle diameters $D$, with one single receive coil in (a) 1D, (b) 2D, and (c) 3D domains.
 }
 \label{fig:sv_FFP_sin_single}
\end{figure}

Next we consider the FFL case. We use a field free line in affine planes of the $e_1$-$e_2$-plane,
following \cite{Knopp2011FFLfourierslice}, where $e_i$ denotes the $i$th canonical Cartesian
coordinate. 
We use a dynamic selection field $g: \R^3 \times I \rightarrow \R^3$, i.e., $g(x,t)=P(t)^tGP(t)x$, with
$P:I\rightarrow \R^{3\times 3}$ given by
\begin{equation*}
 P(t)=\begin{pmatrix}
       \cos(2\pi f_\mathrm{rot} t) & \sin(2\pi f_\mathrm{rot} t) & 0 \\ -\sin(2\pi f_\mathrm{rot} t) & \cos(2\pi f_\mathrm{rot} t)& 0 \\ 0& 0 &1     \end{pmatrix}
\end{equation*}
for $f_\mathrm{rot} >0$. A translational movement is performed by using a drive field perpendicular to the FFL
at any time $t\in I$. Assuming $G_{1,1}=0$, the non-rotated FFL is the $e_1$-axis and we obtain the rotated drive field
\begin{equation*}
 h(t)=-P(t)^t \tilde{h}(t) =-A_2 \sin(2 \pi f_2 t) ( - \sin( 2 \pi f_\mathrm{rot} t) , \cos( 2 \pi f_\mathrm{rot} t) , 0)^t   - A_3 \sin(2 \pi f_3 t) e_3.
\end{equation*}
for a non-rotated drive field $\tilde{h}(t)=(0,A_2 \sin(2 \pi f_2 t),A_3 \sin(2 \pi f_3 t)$, $A_2,A_3,f_2,f_3>0$.
The parameters used to obtain the FFL results are given in Table \ref{tab:sim-parameters}. We use measurement times
comparable to the FFP case. The rotation frequency $f_\mathrm{rot}$ is chosen such that two complete rotations are
performed during the whole measurement time in 2D. In 3D, a slowly varying field moves the FFL plane in $e_3$-direction.
The corresponding frequency $f_3$ is chosen such that the field's period is the whole measurement time.
Since the rotation matrix $P(t)$ depends smoothly on time $t$, it does not influence the analysis in Theorem
\ref{thm:limit2}. The numerical results are given in Fig. \ref{fig:sv_FFL_sin}. It is observed that the
SVs decay slower, when the particle diameter $D$ increases. In 2D we observe a decay rate slightly slower than
$O(n^{-\frac12})$ for the leading SVs for the largest diameter. Qualitatively, the 3D limit case agrees with
the prediction in Theorem \ref{thm:limit2} but it does not completely approach $O(n^{-\frac14})$
(the sharp one is conjectured to be $O(n^{-\frac12})$; see Remark \ref{rem:conjecture2}).

\begin{figure}[hbt!]
\centering
 	\begin{subfigure}[t]{0.35\textwidth}
 		\includegraphics[width=\textwidth,trim={1.4cm 0 4.8cm 0},clip]{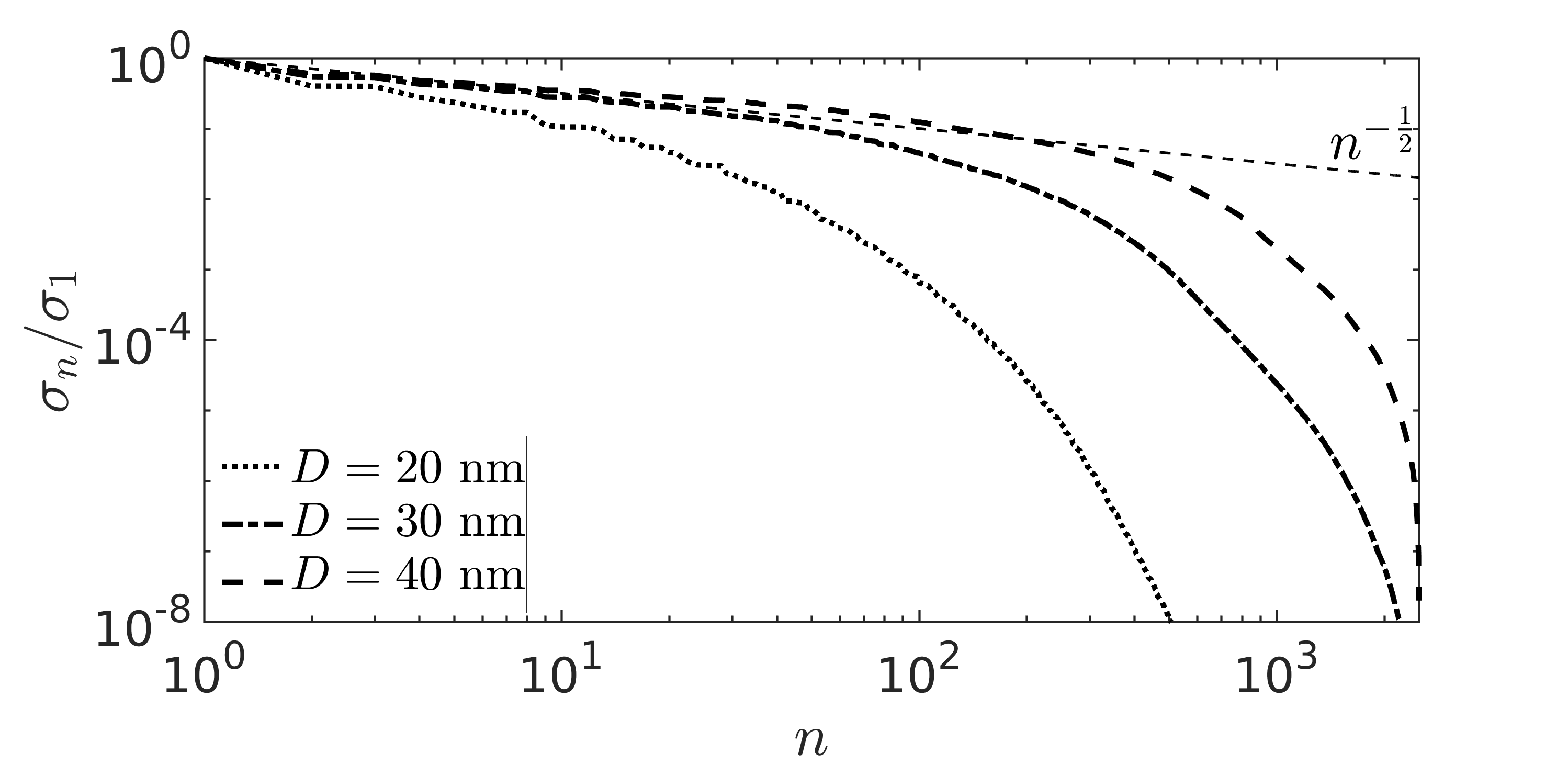}
		\caption{FFL 2D  }
		\label{subfig1:sv_FFL_sin}
	\end{subfigure}
	 	\begin{subfigure}[t]{0.35\textwidth}
 		\includegraphics[width=\textwidth,trim={1.4cm 0 4.8cm 0},clip]{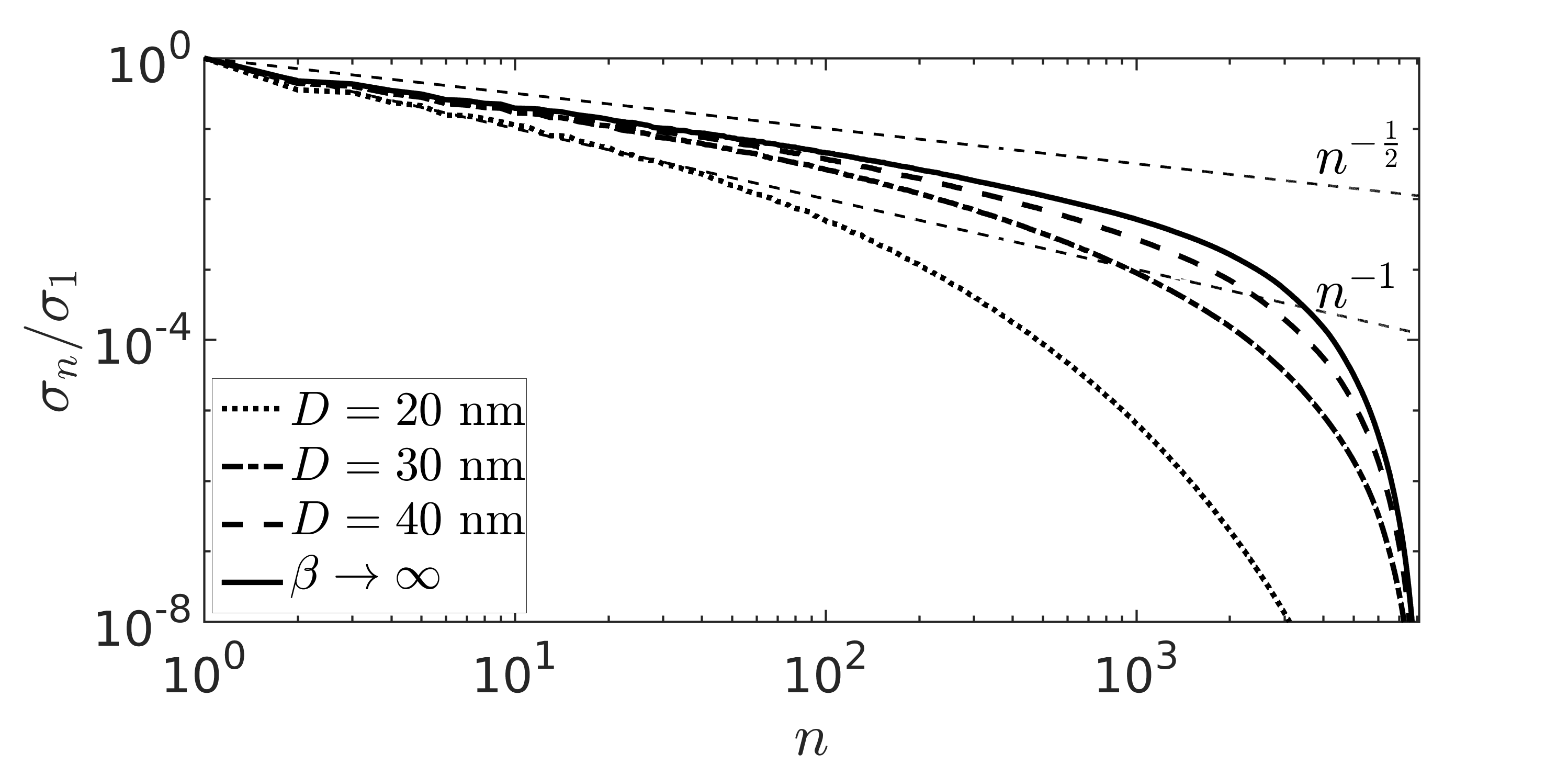}
		\caption{FFL 3D  }
		\label{subfig2:sv_FFL_sin}
	\end{subfigure}
 \caption{SV decay for sinusoidal FFL excitation patterns and different particle diameters in (a) 2D and (b) 3D domains.}
 \label{fig:sv_FFL_sin}
\end{figure}

Last, we compare the FFP and FFL cases. The comparative results in Fig. \ref{fig:sv_FFL_vs_FFP} show
clearly the slower SV decays for the leading SVs for the FFL case
for one example particle diameter in 2D and the limit case in 3D.
Equivalently, for a given noise level, there are more significant singular values in the FFL case
than in the FFP case, which potentially allows obtaining higher resolution reconstructions.

\begin{figure}
\centering
 	\begin{subfigure}[t]{0.35\textwidth}
 		\includegraphics[width=\textwidth,trim={1.4cm 0 4.8cm 0},clip]{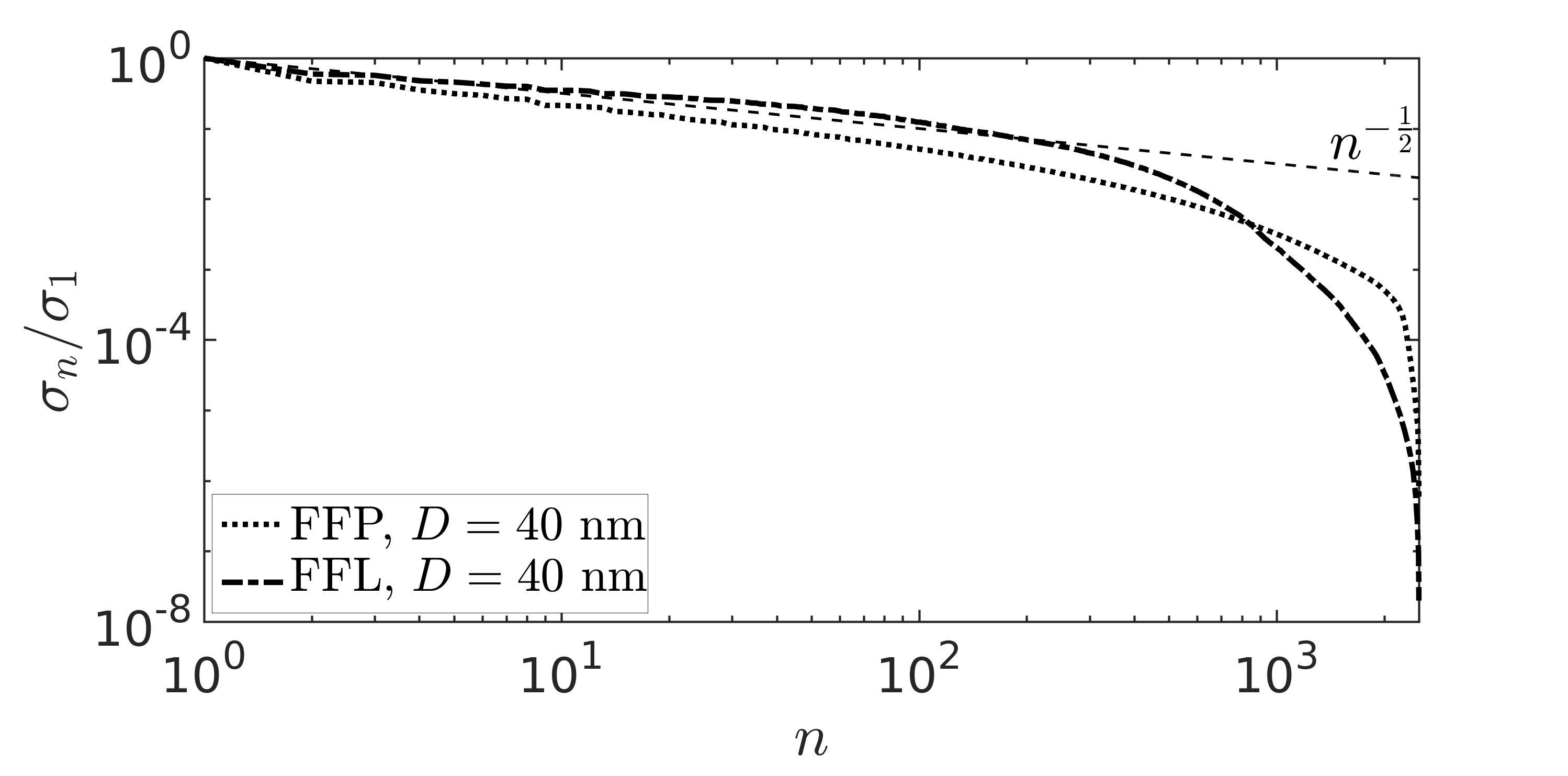}
		\caption{FFP vs. FFL 2D  }
		\label{subfig1:sv_FFL_vs_FFP}
	\end{subfigure}
	 	\begin{subfigure}[t]{0.35\textwidth}
 		\includegraphics[width=\textwidth,trim={1.4cm 0 4.8cm 0},clip]{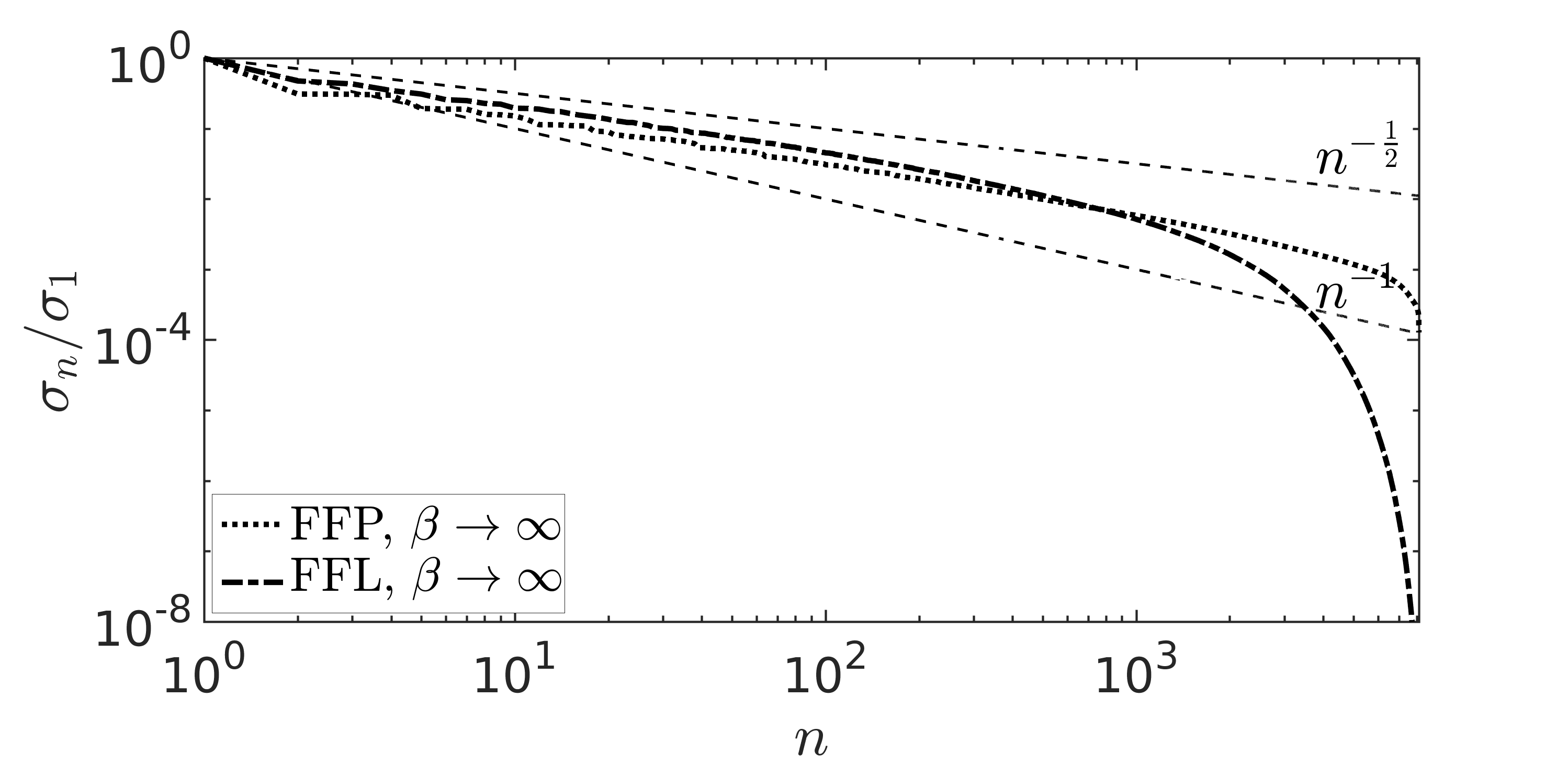}
		\caption{FFP vs. FFL 3D  }
		\label{subfig2:sv_FFL_vs_FFP}
	\end{subfigure}
 \caption{The comparison of SV decay for the FFP and FFL cases with sinusoidal excitation patterns in (a) 2D and (b) 3D domains. }
 \label{fig:sv_FFL_vs_FFP}
\end{figure}

\section{Discussions and concluding remarks}\label{sec:conc}
In this work, we have analyzed the nonfiltered MPI equilibrium model with common experimental setup, 
and studied the degree of ill-posedness of the forward operator via SV decay for Sobolev smooth bivariate functions. Our analysis gives rise to the
following findings. The standard setup in MPI using trigonometric drive field patterns and a linear selection
field leads to a severely ill-posed problem. For the nonfiltered model, even if the trajectories $h(t)$ are
nonsmooth, for the linear selection field, Theorem \ref{thm:non-filter2}
predicts an exponential SV decay. The resolution
improvement for larger diameters reported in \cite{Weizenecker2007,knopp2008singular} can be explained by
considering the limit case. Two different MPI methodologies, i.e., FFP and FFL, are distinguished in this work, where the FFL case has not been studied theoretically before.
The FFL approach can potentially lead to a less ill-posed problem than the FFP approach (Theorems
\ref{thm:limit1} and \ref{thm:limit2}). In the discrete setup, the temporal behavior of the FFL 
needs to be chosen more carefully to fully exploit the potential benefits predicted by the limit case. In
particular, this has to be considered when parameterizing the FFL case with respect to hardware limitations,
e.g., rotation frequency \cite{Bakenecker2017}. Further, a theoretical result for the filtered problem exhibits
a faster SV decay for analog filters with high (temporal) regularity.

The theoretical findings in this work build the basis for several directions of further research of
both theoretical and empirical nature. First, the predicted severe ill-posedness for small particle diameters
opens the avenue for developing more efficient algorithms based on low-rank approximations of the
forward operator. Second, the techniques are fairly general and might also be employed to
analyze other more refined models for MPI, which include particle dynamics and particle-particle
interactions \cite{Kluth:2017}. Last, for image
reconstruction, it is natural to analyze regularized formulations
in the context of regularization theory \cite{EnglHankeNeubauer:1996,ItoJin:2015}.

\section*{Acknowledgements}
The authors are grateful to the referee and the board member for constructive comments, which have
improved the quality of the work. T. Kluth is supported by the Deutsche Forschungsgemeinschaft (DFG)
within the framework of GRK 2224/1 ``Pi$^3$ : Parameter Identification - Analysis, Algorithms,
Applications'', and G. Li partially by Hausdorff Center for Mathematics, University of Bonn and
a Newton international fellowship from Royal Society.

\bibliographystyle{abbrv}
\bibliography{literature}

\end{document}